\documentclass[11pt,reqno]{amsart}
\usepackage {amssymb}
\usepackage {amsmath}
\usepackage {bbm}
\usepackage{amsthm}
\usepackage{mathtools}
\usepackage{graphicx}
\usepackage {amscd}
\usepackage {epic}
\usepackage[mathscr]{euscript}
\usepackage{comment}

\usepackage{geometry}
\usepackage{bm}

\usepackage[dvipsnames]{xcolor}
\usepackage[colorlinks,citecolor=OliveGreen,linkcolor=Mahogany,urlcolor=Plum,pagebackref]{hyperref}
\usepackage[alphabetic]{amsrefs}

\usepackage{todonotes}

\theoremstyle{plain}
\newtheorem{theorem}{Theorem}[section]
\newtheorem{lemma}[theorem]{Lemma}
\newtheorem{proposition}[theorem]{Proposition}
\newtheorem{corollary}[theorem]{Corollary}

\newtheorem{cor}[theorem]{Corollary}

\newtheorem{prop}[theorem]{Proposition}
\newtheorem{lem}[theorem]{Lemma}

\theoremstyle{definition}
\newtheorem{example}[theorem]{Example}
\newtheorem{definition}[theorem]{Definition}

\newtheorem{defn}[theorem]{Definition}

\theoremstyle{remark}
\newtheorem{remark}[theorem]{Remark}

\newcommand{\cF}{\mathcal{F}}

\newcommand{\cL}{\mathcal{L}}
\newcommand{\cO}{\mathcal{O}}

\newcommand{\cS}{\mathcal{S}}

\newcommand{\bN}{\mathbb{N}}

\newcommand{\bQ}{\mathbb{Q}}

\newcommand{\bR}{\mathbb{R}}

\newcommand{\bZ}{\mathbb{Z}}

\newcommand{\fB}{\mathfrak{B}}

\newcommand{\fa}{\mathfrak{a}}
\newcommand{\fb}{\mathfrak{b}}
\newcommand{\fm}{\mathfrak{m}}
\newcommand{\fc}{\mathfrak{c}}
\newcommand{\fd}{\mathfrak{d}}

\newcommand{\supp}{\mathrm{supp}}

\newcommand{\vol}{\mathrm{vol}}
\newcommand{\ord}{\mathrm{ord}}

\newcommand{\lct}{\mathrm{lct}}
\newcommand{\Int}{\mathrm{Int}}
\newcommand{\Val}{\mathrm{Val}}
\newcommand{\DivVal}{\mathrm{DivVal}}
\newcommand{\nvol}{\widehat{\mathrm{vol}}}

\newcommand{\Spec}{\mathrm{Spec}}

\newcommand{\gr}{\mathrm{gr}}

\newcommand{\QM}{\mathrm{QM}}

\newcommand{\e}{\mathrm{e}}

\newcommand{\red}{\mathrm{red}}

\newcommand{\R}{\mathbb{R}}
\newcommand{\la}{\lambda}

\numberwithin{equation}{section}

\begin{document}

	\title{Convexity of multiplicities of filtrations on local rings}
	

	
	\author{Harold Blum}
	
\address{Department of Mathematics\\
  University of Utah\\
Salt Lake City, UT 84112, USA.}
\email{blum@math.utah.edu}
	
	\author{Yuchen Liu}
	\address{Department of Mathematics, Northwestern University, Evanston, IL 60208, USA.}
\email{yuchenl@northwestern.edu}

	\author{Lu Qi}
	\address{Department of Mathematics, Princeton University, Princeton, NJ 08544, USA.}
	\email{luq@princeton.edu}

\begin{abstract}
We prove that the multiplicity of a filtration of a local ring satisfies various convexity properties. 
In particular, we show the multiplicity is convex along geodesics. 
As a consequence, we prove that the volume of a valuation is log convex on simplices of  quasi-monomial valuations 
and give a new proof of a theorem of Xu and Zhuang on the uniqueness of normalized volume minimizers. 
In another direction, we generalize a theorem of Rees on multiplicities of ideals to filtrations 
and characterize when the Minkowski inequality for filtrations is an equality under mild assumptions.

\end{abstract}
	
\maketitle

\setcounter{tocdepth}{1}	
\tableofcontents

\section{Introduction}
	
\begin{center}
\emph{Throughout the article, $(R,\fm)$ denotes an $n$-dimensional, analytically irreducible, Noetherian, local domain.}
\end{center}

\medskip

The Hilbert-Samuel multiplicity of an $\fm$-primary ideal $\fa\subset R$ is a fundamental invariant of the singularities of $\fa$ and satisfies various convexity properties such as Teissier's Minkowski inequality \cite{Tei78}.
In this article, we consider $\fm$-filtrations of $R$, 
which generalize the filtrations of $R$ given by the powers of a single ideal, 
and prove various convexity properties for such filtrations.
The results have applications to the study of K-stability, volumes of valuations, and problems in commutative algebra. 

An \emph{$\fm$-filtration}  is a collection $\fa_\bullet=(\fa_\la)_{\la\in \bR_{>0}}$ of $\fm$-primary ideals of $R$ that is decreasing, graded, and left continuous.
The latter three conditions mean $\fa_\la \subset \fa_{\mu}$ when $\la>\mu$, $\fa_{\la}\cdot \fa_{\mu}\subset \fa_{\la+\mu}$, and $\fa_{\la}= \fa_{\la-\epsilon}$ when $0<\epsilon\ll1$, respectively.
The key examples of  $\fm$-filtrations are the following:
\begin{enumerate}
\item A trivial example is given by taking powers  $( \fb^{\lceil \lambda \rceil })_{\lambda \in \bR_{>0}}$ of a fixed $\fm$-primary ideal $\fb\subset R$.
\item An important example in this paper is $\fa_{\bullet}(v)\coloneqq(\fa_{\la}(v))_{\la\in \bR_{>0}}$, where $v\colon {\rm Frac}(R)^\times  \to \mathbb{R}$ is a valuation  centered at $\fm$
and $\fa_{\la}(v) \coloneqq \{f \in R\, \vert \, v(f) \geq \la\}$  (see Section \ref{sec:valuation}).
\item If $(\fb_{\la})_{\la\in \bZ_{>0}}$ is a decreasing, graded sequence of $\fm$-primary ideals,\footnote{In \cite{Cut21}, an $\fm$-filtration is indexed by $\bZ_{>0}$ and, hence, coincides with a decreasing graded sequence of $\fm$-primary ideals.}
 then $(\fb_{\lceil \la \rceil })_{\la\in \bR_{>0}}$ is an $\fm$-filtration.
 Such sequences have been well studied in the literature \cite[Section 2.4.B]{Laz04}.
 
\end{enumerate}

Following work of Ein, Lazarsfeld, and Smith, the \emph{multiplicity} of an $\fm$-filtration is
\[
\vol(\fa_\bullet)
\coloneqq 
\lim_{ \bN \ni m\to\infty} \frac{\ell(R/\fa_m)}{m^n/n!} 
= 
\lim_{ \bN \ni m\to\infty} \frac{\e(\fa_m)}{m^n} ,
\]
where the existence of the above limit and the equality were proven in increasing generality by \cite{ELS03,Mus02,LM09,Cut13,Cut14}. 
This invariant is the local counterpart of the volume of a graded linear series of a line bundle
and has been studied both in the context of commutative algebra \cite{CSS19,Cut21} and  recently  in work of C. Li and others on the normalized volume of a valuation \cite{Li18,LLX18}.
	
In this article, we prove two properties of the multiplicity of $\fm$-filtrations. 
The first is a convexity result, which has applications to volumes of valuations and C. Li's normalized volume function.
The second is a generalization of a classical theorem of Rees \cite{Ree61} from the setting of ideals to filtrations and results in a characterization of when the Minkowski inequality for filtrations is an equality.

\subsection{Multiplicity  and geodesics}
Given two $\fm$-filtrations $\fa_{\bullet,0}$ and $\fa_{\bullet,1}$,
we define a segment of $\fm$-filtrations $(\fa_{\bullet,t})_{t\in[0,1]}$ interpolating between $\fa_{\bullet,0}$ and $\fa_{\bullet,1}$ by setting
\[
\fa_{\lambda,t} 
\coloneqq  
\sum_{\lambda=(1-t)\mu+t\nu} \fa_{\mu,0} \cap \fa_{\nu,1}
.\]
We call $(\fa_{\bullet,t})_{t\in[0,1]}$ the \emph{geodesic} between $\fa_{\bullet,0}$ and $\fa_{\bullet,1}$, since it is the local analogue of the geodesic between two filtrations of the section ring of a polarized variety \cite{BLXZ21,Reb20}.
This definition is also related to a construction in \cite{XZ20b}.

In \cite{BLXZ21}, it was shown that several non-Archimedean functionals from the theory of K-stability are strictly convex along geodesics in the global setting. 
In a similar spirit, we prove a convexity result in the local setting for the multiplicity along  geodesics.
	
\begin{theorem}\label{thm:main theorem}
Assume $R$ contains a field.
If  $\fa_{\bullet,0}$ and $\fa_{\bullet,1}$ are  $\fm$-filtrations with positive multiplicity, then the function $E(t) \colon [0,1]\to \bR$ defined by $E(t) \coloneqq  \e(\fa_{\bullet,t})$
satisfies the following properties:
\begin{enumerate}
\item $E(t)$ is smooth;
\item $E(t)^{-1/n}$ is concave, meaning 
\begin{equation*}\label{eqn:E^1/n is concave}
E(t)^{-1/n}\ge (1-t)E(0)^{-1/n}+tE(1)^{-1/n} \quad \text{ for all } t\in [0,1];
\end{equation*}
\item $E(t)^{-1/n}$ is linear if and only if $\widetilde{\fa}_{\bullet,0} =\widetilde{\fa}_{c\bullet,1}$ for some $c\in \bR_{>0}$.
\end{enumerate}
\end{theorem}

The term smooth in Theorem \ref{thm:main theorem}.1
means $E(t)$ extends to a $C^\infty$ function on $(-\epsilon,1+\epsilon)$ for some $\epsilon>0$.
The symbol $\widetilde{\, \cdot \,}$  in Theorem \ref{thm:main theorem}.3  denotes the saturation of an $\fm$-filtration, which is  defined in Section  \ref{ss:defsat}.
This notion is an analogue of the integral closure of an ideal in the setting of filtrations and discussed further in Section \ref{ss:Rees} below. 

The proof of Theorem \ref{thm:main theorem} is inspired by a related argument in the global setting \cite{BLXZ21} and relies on constructing a measure on $\mathbb{R}^2$ that encodes the multiplicities of the filtrations along the geodesic.
In the special case when $\fa_{\bullet,0}$ is the $\fm$-filtration of a valuation minimizing the normalized volume function over a klt singularity, the proof of \cite{XZ20b} can be used to show Theorem \ref{thm:main theorem}(2). 
Theorem \ref{thm:main theorem} removes these strong restrictions and is proven without the theory of K-stability for valuations introduced in \cite[Section 3.1]{XZ20b}.

\subsubsection{Applications to volume}  
As first defined by Ein, Lazarsfeld, Smith,
the \emph{volume} of a valuation $v\colon {\rm Frac}(R)^{\times} \to \mathbb{R}$  centered at $\fm$ is 
\[
\vol(v) \coloneqq \e(\fa_\bullet(v)) = \lim_{m \to \infty} \frac{ \ell \left(R/ \fa_{m}(v)\right) }{m^n/n!}
.\]
This invariant is a local analogue of the volume of a line bundle and also plays a role in the study of K-stability of Fano varieties  and Fano cone singularities.

Theorem \ref{thm:main theorem} can be applied to show that  the volume of a valuation is strictly log convex on simplices of quasi-monomial valuations in the valuation space of $(R,\fm)$.
This gives an affirmative answer to \cite[Question 6.23]{LLX18}. 
We note that the  volume was previously shown to be Lipschitz continuous on such a simplex  \cite[Corollary D]{BFJ14}. 
	
\begin{corollary}[Convexity of volume]\label{cor:convexityofvolume}
Assume $R$ contains a field. Let $\eta \in (Y,D_1+\cdots +D_r)$ be a log smooth birational model of $(R,\fm)$. 
For any $\bm{\alpha},\bm{\beta} \in \bR_{>0}^{r}$ and $t\in (0,1)$, 
\[
\vol(v_{(1-t)\bm{\alpha}+ t \bm{\beta}})^{-1/n}
\geq 
(1-t)\vol(v_{\bm{\alpha}})^{-1/n} + t \vol(v_{\bm{\beta}})^{-1/n} 
\] 
and equality holds if and only if $\bm{\alpha} = c \bm{\beta}$ for some $c\in \bR_{>0}$.
\end{corollary}

In the above theorem, $v_{\bm{\alpha}}$ denotes the quasi-monomial valuation of ${\rm Frac}(R)$ with weights $\bm{\alpha}$ on $D_1,\ldots, D_r$. 
See Section \ref{sssec:QM} for a detailed definition.

\subsubsection{Applications to normalized volume}
In \cite{Li18}, Chi Li defined the normalized volume of a valuation over a klt singularity and proposed the problem of studying its minimizer.
The notion  plays an important role in the study of K-stability of Fano varieties and
in the study of klt singularities. 
The invariant has  been extensively studied in the recent years; see \cite{LLX18} and \cite{Zhu23} for  surveys on this topic.

The fundamental problem in the study of the normalized volume function is the Stable Degeneration Conjecture proposed by Li  \cite[Conjecture 7.1]{Li18} and Li and Xu \cite[Conjecture 1.2]{LX18}. 
The conjecture predicts that there exists a valuation minimizing the normalized volume function and that the minimizer is unique up to scaling, quasi-monomial, 
has finitely generated associated graded ring, and induces a degeneration of the klt singularity to a K-semistable Fano cone singularity. 
These five statements were proven in  \cite{Blu18,LX18,Xu20,XZ20b,XZ22}. In particular, the Stable Degeneration Conjecture is now a theorem.

Using Theorem \ref{thm:main theorem}, we revisit 
Xu and Zhuang's theorem stating that the  minimizer of the normalized volume function is unique up to scaling  \cite{XZ20b}. (The uniqueness was also previously proven in \cite{LX18} under the assumption that the minimizer has finitely generated associated ring. The latter was recently shown in \cite{XZ22}.) 
In particular, we give a proof of the uniqueness result independent of the theory of K-semistability for valuations developed in  \cite{XZ20b}.
	
\begin{corollary}[{Uniqueness of minimizer}]\label{cor:uniqueness of minimizer}
If  $x\in(X,D)$ is a klt singularity defined over an algebraically closed field of characteristic 0, then any minimizer of $\nvol_{X,D,x}$ (see Definition \ref{d:nvol}) is unique up to scaling. 
\end{corollary}

The new proof of Corollary \ref{cor:uniqueness of minimizer} takes the following direct approach.
Fix two valuations $v_0$ and $v_1$ that minimize $\nvol_{X,D,x}$ and consider the geodesic $(\fa_{\bullet,t})_{t\in [0,1]}$ between
$\fa_{\bullet}(v_0)$ and $\fa_{\bullet}(v_1)$. 
Using Theorem \ref{thm:main theorem}(2), a characterization of the infimum of the normalized volume function in terms of normalized multiplicities \cite[Theorem 27]{Liu18}, and an inequality of log canonical thresholds \cite[Theorem 3.11]{XZ20b}, 
we show that $ \e(\fa_{\bullet,t})$ is linear. 
Theorem \ref{thm:main theorem}(3)  then implies  $cv_0=v_1$ for some $c>0$.

\subsection{Rees's theorem}\label{ss:Rees}
A theorem of Rees \cite{Ree61} states that if $\fa\subset \fb$ are two $\fm$-primary ideals, then the following statements are equivalent:
\begin{enumerate}
\item $\e(\fa)=\e(\fb)$.
\item $\overline{ \bigoplus_{m\in \bN} \fa^m}=\overline{ \bigoplus_{m \in \bN} \fb^m}$.
\item $\overline \fa=\overline \fb$.
\end{enumerate}
The symbol $\overline{\,  \cdot\, }$ in (2)  denotes the algebraic closure in $R[t]$, while in (3) it denotes the integral closure of an ideal.
The equivalence between (2) and (3) follows from definitions. 

It is natural to ask for a generalization of the above result for $\fm$-filtrations  $\fa_\bullet\subset \fb_\bullet$.  
In \cite{Cut21}, Cutkosky studies whether the two conditions
\begin{enumerate}
\item $\e(\fa_\bullet)=\e(\fb_\bullet)$ 
\item $\overline{ \bigoplus_{m\in \bN} \fa_m}=\overline{ \bigoplus_{m \in \bN} \fb_m}$
\end{enumerate}
are equivalent.
While (2) $\Rightarrow$ (1) holds by \cite[Theorem 6.9]{CSS19},
(1) $\Rightarrow$ (2) can fail even in very simple examples (see Example \ref{eg:simple eg}). 
That said, (1) $\Rightarrow$ (2)  holds for special classes of $\fm$-filtrations \cite[Theorem 1.4]{Cut21}.

To remedy this issue, we introduce the \emph{saturation} $\widetilde{\fa}_\bullet$ of an $\fm$-filtration $\fa_\bullet$ in Section \ref{sec:saturation}. 
(The definition may be viewed as a local analogue of a construction studied by Boucksom and Jonsson in \cite{BJ21}; see Section \ref{sec:global}).
The saturation is defined using divisorial valuations, analogous to the valuative definition of the integral closure of an ideal. 
Using this notion, we prove a version of Rees's theorem for filtrations.

\begin{theorem}[{Rees's Theorem}]\label{thm:characterization for sequences with equal volume}
For $\fm$-filtrations  $\fa_\bullet \subset \fb_\bullet$, 
$\e(\fa_\bullet)=\e(\fb_\bullet)$ if and only if $\widetilde\fa_\bullet=\widetilde\fb_\bullet$. 
\end{theorem}

The above result can be explained as follows. 
The multiplicity of an $\fm$-filtration is determined by its valuative properties, not by the integral properties of its Rees algebra.
These  two properties coincide for ideals, but do not always coincide  for filtrations by Example \ref{eg:simple eg}.
	
	
\subsection{Minkowski inequality}
By work of Teissier \cite{Tei78}, Rees and Sharp \cite{RS78}, and D. Katz \cite{Kat88} in increasing generality, 
for two $\fm$-primary ideals $\fa$ and $\fb$ of $R$, 
\[
\e(\fa\fb)^{1/n}\leq \e(\fa)^{1/n}+\e(\fb)^{1/n}
\]
and the equality holds if and only if there exist $c,d\in \bZ_{>0}$ such that  $\overline{\fa^c}=\overline{\fb^d}$.

For two $\fm$-filtrations $\fa_\bullet$ and $\fb_\bullet$,  we let $\fa_\bullet\fb_\bullet$ denote the $\fm$-filtration $( \fa_\la  \fb_\la)_{\la\in \bR_{>0}}$.
Using the saturation of a filtration, we characterize when the Minkowski inequality for filtrations is an equality.

\begin{corollary}[Minkowski Inequality]\label{cor:Minkowski}
For $\fm$-filtrations $\fa_\bullet$ and $\fb_\bullet$ with positive multiplicity,
\begin{equation}\label{eqn:Minkowski-Teissier}
\e(\fa_\bullet\fb_\bullet)^{1/n}\leq
\e(\fa_\bullet)^{1/n}+\e(\fb_\bullet)^{1/n}
\end{equation}
and the equality holds if and only if $\widetilde\fa_{\bullet} = \widetilde\fb_{c\bullet}$  for some $c\in \bR_{>0}$. 
\end{corollary}

The inequality statement of the above corollary is not new and is due to
Mustaţǎ \cite{Mus02}, Kaveh and Khovanskii \cite{KK14}, and Cutkosky \cite{Cut15} in increasing levels of generality.
In the equality statement, the forward implication follows easily from \cite[Theorem 10.3]{Cut21} and the definition of the saturation, while the reverse implication relies on Theorem \ref{thm:characterization for sequences with equal volume}.

\begin{remark}[Relation to work of Cutkosky]
Cutkosky proved a version of the equality part of 
Corollary \ref{cor:Minkowski}
in the special case when  $\fa_\bullet$ and $\fb_\bullet$ are \emph{bounded} filtrations \cite[Definition 1.3]{Cut21}, which roughly means that their integral closure is induced by a finite collection of divisorial valuations \cite[Theorem 1.6]{Cut21}.
For such filtrations,   the integral closure and saturation agree by Lemma \ref{lem:valuative ideals sat}.
Thus, Corollary \ref{cor:Minkowski} may be viewed as a generalization of the latter result.
Similarly, Theorem  \ref{thm:characterization for sequences with equal volume} may be viewed as a  generalization of \cite[Theorem 1.4]{Cut21}.
\end{remark}
	 
\medskip

This paper is organized as follows:  In Section \ref{sec:prelim}, we recall definitions and basic facts concerning filtrations, valuations, and multiplicities.  
In Section \ref{sec:saturation}, we introduce the saturation of a filtration and prove Theorem \ref{thm:characterization for sequences with equal volume}.
In Section \ref{sec:multandgeodesics}, we define the geodesic between two filtrations and prove Theorem \ref{thm:main theorem}.
In Section \ref{sec:proof}, we deduce Corollaries \ref{cor:convexityofvolume}, \ref{cor:uniqueness of minimizer}, and \ref{cor:Minkowski} as consequences of results in the previous two sections.
In Section \ref{sec:global}, we discuss relations between the results in this paper and global results in the K-stability literature.
The appendix  of the paper is devoted to an alternate proof of a special case of Theorem \ref{thm:main theorem} using the theory of Okounkov bodies. 
 
\medskip
	
\noindent\textbf{Acknowledgement.} 
We would like to thank Mattias Jonsson, Linquan Ma, Sam Payne, Longke Tang, Chenyang Xu, and Ziquan Zhuang for helpful discussions. 
LQ would like to thank his advisor Chenyang Xu for constant support, encouragement and numerous inspiring conversations.
We are also grateful to the referees whose comments and corrections improved the quality of this paper.  

HB is partially supported by NSF grant DMS-2200690.
YL is partially supported by  NSF Grant DMS-2148266.

\section{Preliminaries}\label{sec:prelim}
Throughout this section, $(R,\fm,\kappa)$ denotes an $n$-dimensional, analytically irreducible,\footnote{Recall that $(R,\fm)$ is \emph{analytically irreducible} if its completion $\hat{R}$ is a domain. The latter holds for example when $R$ is an excellent Noetherian local ring that is normal by \cite[\href{https://stacks.math.columbia.edu/tag/0C23}{Lemma 0C23}]{stacks-project}.} Noetherian, local domain.
We set $X\coloneqq  \Spec(R)$ and write $x\in X$ for the closed point corresponding to $\fm$.

		
\subsection{Filtrations}
\begin{definition}
An $\fm$-\emph{filtration} is a collection $\fa_\bullet =(\fa_\la )_{\la \in \R_{>0}}$  of $\fm$-primary ideals of $R$ such that
\begin{enumerate}
\item $\fa_\la \subset \fa_{\mu}$ when  $\la >\mu$,
\item $\fa_{\la} = \fa_{\la-\epsilon}$ when $0<\epsilon\ll1$, and
\item $\fa_{\la} \cdot \fa_{\mu} \subset \fa_{\la+ \mu}$ for any $\lambda,\mu\in\bR_{>0}$.
\end{enumerate}
By convention, we set $\fa_{0}\coloneqq R$.  The definition is a local analogue of a filtration of the section ring of a polarized variety in \cite{BHJ17}.
\end{definition}

For $\lambda\in\bR_{\ge 0}$, define $\fa_{>\lambda}:=\bigcup_{\mu>\lambda}\fa_\mu$. If $\lambda$ satisfies $\fa_{>\lambda}\subsetneq \fa_\lambda$, then we say $\lambda$ is a \emph{jumping number} of the filtration.
 

The \emph{scaling} of an $\fm$-filtration $\fa_\bullet$ by $c\in \bR_{>0}$ is $\fa_{c\bullet}:=(\fa_{c \la } )_{\la\in \bR_{>0}}$.
The \emph{product} of two $\fm$-filtrations $\fa_\bullet$ and $\fb_\bullet$ is  $\fa_\bullet \fb_\bullet :=(\fa_\la \fb_\la)_{\la\in \bR_{>0}}$.
Both are again $\fm$-filtrations.

\begin{definition}
An $\fm$-filtration $\fa_\bullet$ is \emph{linearly bounded} if there exists a constant $c>0$ such that $\fa_{\la} \subset \fm^{\lceil c  \la \rceil }$ for all $\la\in \bR_{>0}$.
\end{definition}

\begin{lemma}\label{lem:linearlybounded}
Let $\fa_\bullet$ and $\fb_\bullet$ be $\fm$-filtrations. If $\fa_\bullet$ is  linearly bounded, then there exists $c\in \bR_{>0}$ such that $\fa_{c\la} \subset  \fb_\la $ for all real numbers $\la \geq 1$
\end{lemma}
	
\begin{proof}
Since $\fa_\bullet$ is  linearly bounded, there exists $c_0 \in \bR_{>0}$ such that  $\fa_{\la} \subset \fm^{ \lceil  c_0\la\rceil } $ for all $\la>0$. 
Since $\fb_1$ is an $\fm$-primary ideal, there exists $d\in \bZ_{>0}$ such that $\fm^{d} \subset \fb_1$. 
If we set $c\coloneqq 2d/c_0$, then
\[
\fa_{ c\la } \subset \fm^{\lceil 2d \la \rceil } \subset \fm^{d \lceil  \la \rceil } \subset  (\fb_1)^{\lceil \la \rceil} \subset \fb_{\lceil \la \rceil } \subset \fb_{\la}
\]
for  all $\la\geq1$. (The second inclusion uses that $d\lceil \la \rceil  \leq d(\la+1) \leq 2d \la \leq \lceil 2 d \la \rceil $).
\end{proof}

\subsection{Valuations}\label{sec:valuation}
Let $(\Gamma,\ge)$ be a totally ordered abelian group. A $\Gamma$-\emph{valuation} of $R$ is a map $v\colon {\rm Frac(R)}^\times\to \Gamma$ such that
\begin{enumerate}
\item $v(fg)=v(f)+v(g)$, and
\item $v(f+g)\ge \min\{v(f),v(g)\}$.
\end{enumerate}
By convention, we set $v(0)\coloneqq \infty$. We say $v$ is  \emph{centered at} $\fm$ if $v\geq 0$ on $R$ and $v>0$ on $\fm\subset R$.

A valuation $v$ of $R$ induces a \emph{valuation ring} $R_{v} : = \{ f \in K \, \vert \, v(f) \geq 0\}$.
We write $\fm_v$ for the maximal ideal of $R_v$ and $\kappa_{v}:=R_v/\fm_v$.  

\subsubsection{Real valuations}
When $\Gamma=\bR$ with the usual order, we say that $v$ is a \emph{real valuation}. We denote by $\Val_{R,\fm}$ the set of real valuations centered at $\fm$.\footnote{In the body of this paper, we will only work with real valuations. In the appendix, $\bZ^n$-valued valuation, where $\bZ^n$ has the lexicographic order, will play a key role.} 
In geometric settings, we will instead denote the set by $\Val_{X,x}$.
	
For $v\in\Val_{R,\fm}$ and $\la\in \bR_{>0}$, we define the \emph{valuation ideal}
\[
\fa_{\la}(v) \coloneqq  \{ f \in R \, \vert \, v(f) \geq \la \}
\]
for each $\la\in R$. Using (1) and (2), one can show $\fa_\bullet(v)$ is an $\fm$-filtration.

For $v\in \Val_{R,\fm}$ and an ideal $\fa\subset R$, set
$v(\fa)\coloneqq \min\{v(f)\mid f\in\fa\}$.
For an $\fm$-filtration $\fa_\bullet$,  set
\[
v(\fa_\bullet)\coloneqq  \lim_{ \bN\in m \to \infty} \frac{ v(\fa_m)}{m} = \inf_{m\in \bN} \frac{v(\fa_m)}{m} , 
\]
where the existence of the limit and second equality is \cite[Proposition 2.3]{JM12}.
	
\subsubsection{Divisorial valuations}
A valuation $v\in \Val_{R,\fm}$ is \emph{divisorial} if 
\[
{\rm tr.deg}_{\kappa}(\kappa_v) =n-1.\]
 We write $\DivVal_{R,\fm}\subset \Val_{R,\fm}$ for the set of such valuations.

Divisorial valuations appear geometrically.  
If $\mu:Y\to X$ is a proper birational morphism with $Y$ normal and  $E\subset Y$ a prime divisor, 
then there is an induced valuation  $\ord_{E} \colon {\rm Frac}(R)^\times \to \bZ$. 
If $\mu(E) =x$ and  $c\in \bR_{>0}$, then $c\cdot \ord_E \in \DivVal_{R,\fm}$.
When $R$ is excellent, all divisorial valuations are of this form; see, for example, \cite[Lemma 6.5]{CS22}.

\subsubsection{Quasi-monomial valuations}\label{sssec:QM}
In the following construction, we always assume $R$ contains a field. Let $\mu\colon Y:= \Spec(S) \to X=\Spec(R)$ be a birational morphism
with $R\to S$  finite type
and $\eta\in Y$  a not necessarily closed point such that $\cO_{Y,\eta}$ is regular and $\mu(\eta) =x$.
Given a  regular system of parameters $y_1,\ldots, y_r$ of $\cO_{Y,\eta}$
and $\bm{\alpha}=(\alpha_1,\cdots, \alpha_r)\in \bR_{\geq 0}^r\setminus {\bf 0}$, we define a valuation $v_{\bf \alpha}$ as follows. 
 For $0\neq f\in \cO_{Y,\eta}$, we can write $f$ in $\widehat{\cO}_{Y,\eta} \simeq k(\eta) [[y_1,\ldots, y_r]]$ as
$\sum_{\bm{\beta}\in \bZ_{\geq 0}^r} c_{\bm{\beta}} {y}^{\bm{\beta}}$ and set 
\[
v_{\bm \alpha}(f)\coloneqq \min \{ \langle \bm{\alpha}, {\bm \beta} \rangle \, \vert \, c_{\bm \beta} \neq 0 \}.
\]
A valuation of the above form is called \emph{quasi-monomial}.

Let $D= D_{1}+\cdots +D_r$ be a reduced divisor on $Y$ such that $y_i=0$ locally defines $D_i$ and $\mu(D_i)= x$ for each $i$. 
We call $\eta\in (Y,D)$ a \emph{log smooth birational model} of $X$. 
We write $\QM_\eta(Y,E)  \subset \Val_{X,x}$ for the set of quasi-monomial valuations that can be described at 
$\eta$ with respect to $y_1,\ldots, y_r$ and note that $\QM_\eta (Y,D) \simeq \bR^r_{\geq 0} \setminus \bm{0}$.

\subsubsection{Izumi's inequality}
The order function $R\setminus 0 \to \bN$ is defined by 
\[
\ord_{\fm}(f)\coloneqq \max \{ k \in \bZ_{\geq 0} \, \vert \, f\in \fm^k \}.\] 
The following version of Izumi's inequality compares $\ord_{\fm}$ to a fixed quasi-monomial valuation.
	
\begin{lemma}\label{lem:Izumi ineq}
Let  $v\in\Val_{R,\fm}$. If (i) $v$ is divisorial or (ii) $R$ contains a field and $v$ is quasi-monomial, then there exists a constant $c>0$ such that
\[
v(\fm) \cdot \ord_\fm(f)\leq v(f)\leq c\cdot\ord_\fm(f).
\]
for all $ f\in R$.
In particular, $\fa_\bullet(v)$ is linearly bounded. 
\end{lemma}


\begin{proof}
Since $f \in  \fm^{\ord_{\fm}(f)}$ by definition, 
$v(\fm) \cdot \ord_{\fm}(f) = v(\fm^{\ord_{\fm}(f)}) \leq v(f)$.
It remains to prove the existence of $c>0$ such that $v(f) \leq c \cdot \ord_{\fm}(f)$ for all $f\in R$.

If (i) holds, the the existence of $c$ follows from Izumi's theorem for divisorial valuations as phrased in \cite[Remark 1.6]{RS14}.\footnote{The application of Izumi's Theorem uses our assumption throughout the paper that $R$ is analytically irreducible.}
If (ii) holds, then there exists a log smooth birational model $\eta \in (Y,D)$ of $x\in X$ and  $\bm{\alpha}   \in \bR^r $
such that $v=v_{\bm{\alpha}}$.
Choose $\bm{\gamma} \in \bZ_{>0}^r$ such that $\bm{\alpha}_i \leq \bm{\gamma}_i$ for each $i=1,\ldots, r$ and consider the valuation $w: = v_{\bm{\gamma}}$. 
We claim that $w\in \DivVal_{R,\fm}$. 
Assuming the claim, then (i) implies that
there exists $c>0$ such that $w(f)
 \leq c \cdot \ord_{\fm}(f)$ for all $f\in R$.
 Since $v(f)\leq w(f)$ for all $f\in R$,  (ii) then follows.

To verify that $w\in \DivVal_{R,\fm}$, note that
 $w: {\rm Frac }(R)^\times \to \bR$ is the composition 
\[
{\rm Frac}(R)^\times 
\lhook\joinrel\longrightarrow {\rm Frac}( \widehat{\cO}_{Y,\eta})^\times \overset{\widehat{w}}{\longrightarrow } \bR,
\]
where $\widehat{w}$ is the valuation that sends  $\sum_{\beta \in \bZ^r_{\geq 0} } c_{\bm{\beta}} y^{\bm{\beta}}$ to $\min \{ \langle \bm{\alpha}, \bm{\beta} \rangle \, \vert\, c_{\bm{\beta}} \neq 0\}$.
Using that $w$ is $\bZ$-valued, a computation (see, for example,  \cite[Proposition 3.7]{JM12}) shows $\widehat{w}$ is a divisorial valuation of $\widehat{\cO}_{Y,\eta}$.
Hence $w$ is a divisorial valuation of $\cO_{Y,\eta}$ by \cite[Proposition 9.3.5]{HS06}.
Now, we compute
\begin{multline*}
{\rm tr.deg}_{\kappa}(\kappa_{w})
=
{\rm tr.deg}_{\kappa}(k(\eta))
+
{\rm tr.deg}_{k(\eta)}(\kappa_v)
=
{\rm tr.deg}_{\kappa}(k(\eta))
+
(\dim \cO_{Y,\eta} -1)
=
\dim R.
\end{multline*}
To see that the last equality holds, note that the complete local ring $\widehat{R}$ is a domain by assumption and thus equidimensional. 
Therefore $R$ is universally catenary by  \cite[\href{https://stacks.math.columbia.edu/tag/0AW6}{Tag 0AW6}]{stacks-project}
and so  the dimension formula  \cite[\href{https://stacks.math.columbia.edu/tag/02II}{Tag 02II}]{stacks-project} gives the last equality.
Therefore, $w$ is divisorial as desired.
 \end{proof}

\subsection{Intersection numbers}
The theory of intersection numbers of line bundles on a proper scheme over an algebraically closed fields was developed in \cite{Kle66}. 
We will use a more general framework developed in \cite[Appendix B]{Kle05}.

\subsubsection{Definition}\label{sss:intdef}
Let $Z$ be a proper scheme over an Artinian ring $\Lambda$. 
For line bundles $\cL_1,\ldots, \cL_r$ on $Z$, the function
\[
\mathbb{Z}^r \ni (m_1,\ldots, m_r ) \to 
 \sum_{i\geq0 } (-1)^i  \ell_{\Lambda} \left( H^i \left(Z, \cL_1^{m_1} \otimes \cdots \otimes \cL_r^{m_r}\right) \right) \in \mathbb{Z}
\]
is a polynomial of degree $\leq \dim Z$ \cite[Theorem B.7]{Kle05}.
The \emph{intersection number} $(\cL_1 \cdot \cdots \cL_r)$ is defined to be the coefficient of $m_1\cdots m_r$ in the above polynomial. When the choice of $\Lambda$ is unclear, we will write ${(\cL_1 \cdot \ldots \cdot  \cL_r)}_\Lambda$.

\subsubsection{Intersections of exceptional divisors}
Let $\mu:Y\to X=\Spec(R)$ be a proper birational morphism with  $Y$ normal.
For Cartier divisors $F_1,\ldots, F_{n-1}$ on $Y$ and a Weil divisor $D: = \sum_{i=1}^r a_i D_i$ on $Y$ with support contained in $Y_\kappa$, we set
\[
F_1 \cdot \ldots \cdot F_{n-1} \cdot D 
:=\sum_{i=1}^r a_i (\cO_{Y}(F_1)\vert_{D_i} \cdot \ldots \cdot \cO_{Y}(F_{n-1})\vert_{D_i})_\kappa 
 .\]
This is well defined, since each prime divisor $D_i$ is proper over $\Spec(\kappa)$.

\begin{prop}\label{p:symmetric}
Assume $R$ is complete. If $F_1, \ldots, F_n$ are Cartier divisors on $Y$ with support contained in $Y_\kappa$, 
then $F_1\cdot  \ldots \cdot  F_n$ is independent of the ordering of the $F_i$.
\end{prop}

Since the intersection product in \ref{sss:intdef} is symmetric, 
$F_1\cdot  \ldots \cdot  F_n$ is independent of the order of $F_1,\ldots, F_{n-1}$.
To deduce the full result, we rely on intersection theory \cite{Ful98}. 

A subtle issue is that the results in  \cite[\S1-18]{Ful98} are stated for schemes of finite type over a field and, hence, do not immediately apply to $Y$.
Fortunately, the Chow group of a scheme of finite type over a regular base scheme can be defined and  the results of \cite[\S2]{Ful98} extend to this setting by \cite[\S20.1]{Ful98} (see also \cite[\href{https://stacks.math.columbia.edu/tag/02P3}{Chapter 02P3}]{stacks-project} for the results in an even more general setting).

\begin{proof}
By Cohen's Structure Theorem, there exists a surjective map 
$A \twoheadrightarrow R$, where $A$ is a regular local ring.
Since the composition
$Y\to  X:= \Spec(R)\to  \Spec(A)$
is finite type, 
the framework of \cite[\S20.1]{Ful98} applies.
Using intersection theory on $Y$ and its subschemes, 
we compute
\begin{multline*}
F_1 \cdot \ldots \cdot F_n  
= \sum_{i=1}^r a_i \int_{D_i} c_1(\mathcal{M}_1)  \cap \cdots \cap c_1(\mathcal{M}_{n-1}) \cap [D_i] \\
= \int_{Y_\kappa } c_1(\mathcal{M}_1)  \cap \cdots \cap c_1(\mathcal{M}_{n-1}) \cap [F_n] 
 = \int_{Y_\kappa} F_1  \cdot \ldots  \cdot F_{n-1} \cdot [F_n],
\end{multline*}
 where $\int_{D_i}$ and $\int_{Y_\kappa}$ denote the degree maps induced by the proper morphisms $D_i \to \Spec(\kappa)$ and $Y_\kappa\to \Spec(\kappa)$ \cite[Definition 1.4]{Ful98} and $\mathcal{M}_i : = \cO_Y(F_i)$. 
The first equality holds by \cite[Example 18.3.6]{Ful98} with the fact that $D_i$ is a proper scheme over $\kappa$,
the second by the equality $[F_n] = \sum_{i=1}^r a_i [D_i]$, and the last by the definition of the first Chern class. 
Since 
\[
F_1 \cdot \ldots \cdot F_{n-1} \cdot [F_n] = F_1 \cdot \ldots  \cdot F_n \cdot [Y] \in { A}_*(Y_\kappa)
\]
and the latter is independent of the order of the $F_i$ by \cite[Corollary 2.4.2]{Ful98}, 
the proposition holds.
\end{proof}

\subsection{Multiplicity}

\subsubsection{Multiplicity of an ideal}
The \emph{multiplicity} of an $\fm$-primary ideal $\fa$ is 
\[
\e(\fa)\coloneqq \lim_{m\to \infty} \frac{\ell(R/\fa^m)}{m^n/n!}
.\]

The following intersection formula for multiplicities commonly appears  in the literature when $x\in X$ is a closed point on a quasi-projective variety \cite[pp. 92]{Laz04}.
In the  generality stated below, it follows from \cite{Ram73}.

\begin{prop}\label{lem:intersection formula for mixed mult}
Let $\fa\subset R$ be an $\fm$-primary ideal and $Y\to  \Spec(R)$  a  proper birational morphism with $Y$  normal. If $\fa \cdot \cO_Y=\cO_Y(-E)$ for  a Cartier divisor $E = \sum_i a_i D_i $, then 
\[
\e(\fa)=   (-E)^{n-1} \cdot E .\]
\end{prop}

\begin{proof}
By \cite[Theorem and Remark 1]{Ram73}, $\e(\fa) = (\cL^{n-1})_{R/\fa}$, where $\cL:= \cO_Y(-E)\vert_E$.
Observe that
\[
(\cL^{n-1})_{R/\fa} 
 = \sum   \ell_{\cO_{E,\xi_i}}( \cO_{E,\xi_i}) (\cL_i^{n-1})_{R/\fa}
 =\sum   a_i (\cL_i^{n-1})_{R/\fa}
,\]
 where $E= \sum a_i D_i$ and $\cL_{i} := \cL\vert_{E_i}$. 
Indeed, the first equality is \cite[Lemma B.12]{Kle05} and the second equality holds, since $\cO_{Y,\xi_i}$ is a DVR and $\cO_{E,\xi_i} = \cO_{Y,\xi}/ (\pi^{a_i}) $, where $\pi$ is a uniformizer of the DVR.

To finish the proof, notice that $D_i$ is defined over $\kappa$, since there is a natural inclusion $D_i \hookrightarrow Y_{\kappa,\red} \to Y_{\kappa}$.
 Since the residue field of the local ring $R/\fa$ is $\kappa$, any finite-length $\kappa$-module $M$ satisfies $\ell_{R/\fa}(M) = \ell_{\kappa}(M)$ by  \cite[\href{https://stacks.math.columbia.edu/tag/02M0}{Lemma 02M0}]{stacks-project}. Thus, 
\[
\sum_{i=1}^r   a_i (\cL_i^{n-1})_{R/\fa}
= 
\sum_{i=1}^r  a_i (\cL_i^{n-1})_{\kappa}
=
(-E)^{n-1} \cdot E
,\]
which completes the proof.
\end{proof}

\subsubsection{Multiplicity of a filtration}
Following \cite{ELS03}, the \emph{multiplicity} of a graded sequence of $\fm$-primary ideals
$\fa_\bullet$ is
\[
\e(\fa_\bullet)\coloneqq \lim_{m\to\infty}\frac{\ell(R/\fa_m)}{m^n/n!} \in [0,\infty).
\]
By  \cite{ELS03,Mus02,LM09,Cut13,Cut14} in increasing generality, the above limit exists and
\begin{equation}\label{eqn:vol=mult}
\e(\fa_\bullet)=\lim_{m\to\infty}\frac{\e(\fa_m)}{m^n} = \inf_{m} \frac{\e(\fa_m)}{m^n};
\end{equation}
see, in particular, \cite[Theorem 6.5]{Cut14}.
Also defined in \cite{ELS03}, the \emph{volume} of a valuation $v\in \Val_{R,\fm}$ is 
\[
\vol(v):= \e(\fa_\bullet(v)).
\]
We set 
\[
\Val_{R,\fm}^+\coloneqq \{v\in\Val_{R,\fm}\mid \vol(v)>0\}\subset \Val_{R,\fm}
.
\]

\begin{prop}\label{prop:qm>0vol}
Let $v\in \Val_{R,\fm}$. If (i) $v$ is divisorial  or (ii) $R$ contains a field and $v$ is quasi-monomial, then $\vol(v)>0$.
\end{prop}

\begin{proof}
By Lemma \ref{lem:Izumi ineq},  there exists $c>0$ such that $ \fa_{\bullet}(v) \subset \fm^{c\bullet}$.
Thus, $\vol(v)\geq \e(\fm^{c\bullet}) = c^{-n}\e(\fm)>0$. 
\end{proof}

\subsection{Normalized volume}\label{ss:nvol}
Assume $R$ is the local ring of a closed point on a algebraic variety defined over an algebraically closed field of characteristic 0.
We say $x\in(X,\Delta)$ is a \emph{klt singularity} if $\Delta$ is an $\bR$-divisor  on $X$ such that $K_{X}+\Delta$ is $\R$-Cartier, and $(X,\Delta)$ is klt as defined in \cite{KM98}.

The following invariant was first introduced in \cite{Li18} and  plays an important role in the study of K-semistable Fano varieties and Fano cone singularities.

\begin{definition}\label{d:nvol}
For a klt singularity $x\in(X,\Delta)$, the \emph{normalized volume function} $\nvol_{(X,\Delta),x}:\Val_{X,x}\to (0,+\infty]$ is defined by
\begin{equation*}
\nvol_{(X,\Delta),x}(v)\coloneqq \left\{\begin{aligned}
&A_{X,\Delta}(v)^n\cdot \vol(v), &\text{ if } A_{X,\Delta}(v)<+\infty\\
&+\infty, &\text{ if } A_{X,\Delta}(v)=+\infty,
\end{aligned}\right.
\end{equation*}
where $A_{X,\Delta}(v)$ is the log discrepancy of $v$ as defined in \cite{JM12,BdFFU15}. 
The \emph{local volume} of a klt singularity $x\in(X,\Delta)$ is defined as
\[
\nvol(x,X,\Delta)\coloneqq \inf_{v\in\Val_{X,x}} \nvol_{(X,\Delta),x}(v).
\]
The above infimum is indeed a minimum by \cite{Blu18, Xu20}.
\end{definition}

\section{Saturation}\label{sec:saturation}
	
Throughout this section, $(R,\fm,\kappa)$ denotes an $n$-dimensional, analytically irreducible, Noetherian, local domain.

\subsection{Saturation}\label{ss:defsat}
Let $\fa\subset R$ be an $\fm$-primary ideal. 
 Recall that the \emph{integral closure} $\overline{\fa}$ of $\fa$ can be characterized valuatively by
\[
\overline\fa=\{f\in R\mid v(f)\ge v(\fa) \text{ for all } v\in\DivVal_{R,\fm}\}.
\]
See, for example, \cite[Theorem 6.8.3]{HS06} and \cite[Example 9.6.8]{Laz04}. 
We define the saturation of an $\fm$-filtration in a similar manner.

\begin{definition}\label{defn:saturation}
The \emph{saturation} $\widetilde{\fa}_\bullet$ of an $\fm$-filtration $\fa_\bullet$ is defined by
\[
\widetilde \fa_\la\coloneqq \{f\in \fm \mid v(f)\ge \la\cdot v(\fa_\bullet) \text{ for all } v\in\DivVal_{R,\fm} \}
\]
for each $\la\in \R_{>0}$.
We say $\fa_\bullet$ is \emph{saturated} if $\fa_\bullet=\widetilde\fa_\bullet$. 
\end{definition}

\begin{remark}
Corollary \ref{prop:saturation alt def} shows that it is equivalent to define the saturation using all positive volume valuations, rather than only divisorial valuations.
\end{remark}

\begin{remark}\label{r:bj21}
The saturation is a local analogue of the maximal norm of a multiplicative norm of the the section ring of a polarized variety, which was defined and studied in \cite{BJ21}. See Section \ref{sec:global}.
\end{remark}

\begin{remark} 
Definition \ref{defn:saturation} differs from the definition of \emph{saturation} used in \cite[Section 2]{Mus02} for monomial ideals, 
which coincides  with the ideals in Lemma \ref{lem:integral closure of sequence} defined using the integral closure of the Rees algebra.
\end{remark}

\begin{remark}
After the first version of this paper was posted on the arXiv, Cutkosky and Praharaj introduced an operation on $\bR$-filtrations that is defined using certain asymptotic Hilbert--Samuel functions (see the definition in  \cite[Theorem 1.3]{CP22}). As shown by \cite[Example 7.2]{CP22}, their operation does not always coincide with the saturation.
\end{remark}

\begin{lemma}[{\cite[Lemma 5.6]{Cut21}}]\label{lem:integral closure of sequence}
If $\fa_\bullet$ be an $\fm$-filtration,
then the integral closure of  ${\rm Rees}(\fa_\bullet) := \bigoplus_{m \in \bN} \fa_m t^m$ in $R[t]$ is given by $\bigoplus_{m\in \bN} {\fa}'_m t^m$, where
\[
{\fa}'_m=\{f\in R\, \mid \, \text{ there exists } r>0 \text{ such that }
f^r\in \overline \fa_{rm}\}.
\]
\end{lemma}
	
\begin{example}\label{eg:simple eg}
In general, the ideals $\widetilde{\fa}_m$ and ${\fa}'_m$ appearing above do not coincide. 
For example, let $R\coloneqq k[x]_{(x)}$ and $\fm=(x)$. 
Consider the $\fm$-filtration $\fa_\bullet$ defined by  $\fa_\la\coloneqq (x^{\lceil \la+1\rceil })$. 
Using that $\ord_{\fm} (\fa_\la) =\lceil \la +1 \rceil $ and $\ord_{\fm} (\fa_\bullet)= 1$, 
we compute 
\[
{\fa}'_m =  (x^{m+1}) \subsetneq (x^m) = \widetilde{\fa}_m.
\]  
for each $m\in \bZ_{>0}$.
In particular, ${\rm Rees}(\fa_\bullet)$ is integrally closed, but $\fa_\bullet$ is not saturated.
\end{example}

The following lemma states basic properties of the saturation.
		
\begin{lemma}\label{lem:Rees is normal} For any $\fm$-filtration  $\fa_\bullet$, the following statements hold:
\begin{enumerate}
\item $\fa_\bullet \subset \widetilde{\fa}_\bullet$,
\item  $v(\fa_\bullet)=v(\widetilde\fa_\bullet)$ for all $v\in \DivVal_{R,\fm}$,  and
\item   $\widetilde\fa_\bullet$ is saturated.
\end{enumerate}
\end{lemma}
	
\begin{proof}
For any $\la\in \bR_{>0}$,
\[
v(\fa_\la) =
\lim_{m \to \infty}
\frac{ v(\fa_\la^m)}{m}
\geq 
\lim_{m\to\infty}
\left(
\frac{ v(\fa_{\lfloor \la m\rfloor })}{\lfloor \la m\rfloor }  \cdot \frac{ \lfloor \la m \rfloor}{m} \right)= \la v(\fa_\bullet)
,
\]
where the inequality uses that $(\fa_{\la})^m \subset \fa_{\la m} \subset \fa_{ \lfloor \la m \rfloor }$.
Therefore,
$\fa_{\la} \subset \widetilde{\fa}_\la$, which is (1).

For (2), note that $
v(\fa_\bullet) \leq  \frac{v(\widetilde{\fa}_m ) }{m}  \leq \frac{v(\fa_m)}{m}$ for each $m \in \bZ_{>0}$,
where the first inequality follows from the definition of $\widetilde{\fa}_m$ and the second from (1). 
Sending $m\to \infty$ gives $ v(\fa_\bullet) \leq   v(\widetilde\fa_\bullet)\leq v(\fa_\bullet)$, which implies (2).  
Statement (3) follows immediately from (2).
\end{proof}

We say two $\fm$-filtrations $\fa_\bullet$ and $\fb_\bullet$ are  \emph{equivalent} if $\widetilde{\fa}_\bullet = \widetilde{\fb}_\bullet$. The following proposition gives a characterization of when two filtrations are equivalent after possible scaling. 

\begin{prop}\label{prop:equiv filts}
Let $\fa_\bullet$ and $\fb_\bullet$ be $\fm$-filtrations and $c\in \R_{>0}$. 
The following statements are equivalent:
\begin{enumerate}
    \item $\widetilde{\fa}_\bullet= \widetilde{\fb}_{c \bullet}$.
    \item $v(\fa_\bullet) = c v(\fb_\bullet)$ for all $v\in \DivVal_{R,\fm}$.
\end{enumerate}
\end{prop}

\begin{proof}
First, note that $v(\fb_{c \bullet}) =c v(\fb_\bullet)$ for all $v\in \Val_{R,\fm}$. 
Therefore, (1) implies (2) follows from Lemma \ref{lem:Rees is normal}(2), while (2) implies (1) follows from the definition of the saturation. 
\end{proof}

\subsection{Saturation and completion}
Let $(\widehat{R},\widehat{\fm})$ denote the $\fm$-adic completion of $(R,\fm)$ and write $\varphi:R\to \hat{R}$ for the natural morphism.
(Note that $\widehat{R}$ is a domain by the assumption throughout the paper that $R$ is analytically irreducible.)
For an $\fm$-filtration $\fa_\bullet$, we set $\fa_\bullet \widehat{R} : = (\fa_\la \widehat{R})_{\la>0}$, 
which is an $\widehat{\fm}$-filtration.

\begin{prop}\label{p:completion-saturation}
If $\fa_\bullet$ is an $\fm$-filtration and  $\fb_\bullet :=\fa_\bullet \widehat{R}$, then $\widetilde{\fa}_\bullet  \widehat{R} = \widetilde{\fb}_\bullet$.
\end{prop}

The proposition shows that completion and saturation commute, as is the case with the integral closure of an ideal \cite[Proposition 1.6.2]{HS06}.
As a consequence of the proposition, many  results regarding saturations reduce to the case when $R$ is  a complete local domain.

\begin{proof}
By \cite[Theorem 9.3.5]{HS06}, there is a bijective map $\varphi_* \colon \DivVal_{\widehat{R},\widehat{\fm}} \to \DivVal_{R,\fm}$
that sends $\hat{v}$ to the valuation $v$ defined by composition 
\[
{\rm Frac}(R)^\times \hookrightarrow {\rm Frac}(\hat{R})^\times \overset{\hat{v}}{\to} \bR
.\]
Note that 
$\hat{v}(\fa  \hat{R}) = v(\fa)$ for any ideal $\fa\subset R$.
 
Fix $\la \in \bR_{>0}$.
Using the previous observation  twice and the definition of the saturation, we compute
\[
\hat{v} (\widetilde{\fa}_\la  \hat{R} )
= v(\widetilde{\fa}_\la) 
\geq 
\la v (\fa_\bullet) = \la \hat{v}(\fb_\bullet) 
.\]
Thus, $\widetilde{\fa}_{\la}  \hat{R} \subset \widetilde{\fb}_\la$.
To prove the reverse inclusion, note that $\fc  R = \widetilde{\fb}_\la$ for some ideal $\fc\subset R$, since $\widetilde{\fb}_\la$ is $\widehat{\fm}$-primary. 
As before, we compute
\[
v(\fc)= \hat{v}(\widetilde{\fb}_\la) \geq \la \hat{v}(\fb_\bullet) =\la  v (\fa_\bullet)  
.\]
 Therefore, $\widetilde{\fb}_\la = \fc  \hat{R} \subset \widetilde{\fa}_\la  \hat{R}$.
\end{proof}

\begin{prop}\label{p:completion-multiplicity}
If $\fa_\bullet$ is an $\fm$-filtration and $\fb_\bullet = \fa_\bullet \widehat{R}$, 
then $\e(\fa_\bullet) = \e(\fb_\bullet)$.
\end{prop}

\begin{proof}
Since $R/\fa_\la$ and $\widehat{R}/\fb_\la$ are isomorphic as Artinian rings, $\ell(R/\fa_\la)=\ell(\widehat{R}/\fb_\la)$ for each $\la>0$. Therefore, the equality of multiplicities holds.
\end{proof}

\subsection{Saturation and multiplicity}\label{ss:satmult}
In this section we prove Theorem \ref{thm:characterization for sequences with equal volume} and a number of corollaries.
The theorem is a consequence of the following two propositions.
		
\begin{proposition}\label{prop:multval}
Let $\fa_\bullet \subset \fb_\bullet$ be  $\fm$-filtrations.
If $\e(\fa_\bullet)= \e(\fb_\bullet)$, then 
$v(\fa_\bullet) = v(\fb_\bullet)$
for all  $v\in \Val_{R,\fm}^+$.
\end{proposition}
	
The proposition was shown when $v$ is divisorial  in \cite[Theorem 7.3]{Cut21} using Okounkov bodies. 
The proof below instead follows the strategy of \cite[Proposition 2.7]{LX16} and \cite[Lemma 3.9]{XZ20b}.

\begin{proof}
Suppose the statement is false. 
Then there exists  $v\in \Val_{R,\fm}^+$ such that $v(\fb_\bullet) < v(\fa_\bullet)$.
After scaling $v$, we may assume $v(\fa_\bullet)=1$. 
Therefore, there exists $l \in \mathbb{Z}_{>0}$ so that $ v(\fb_{l}) /l <v(\fa_\bullet)=1$. 
Hence, we may choose $f\in \fb_l\setminus \fa_l$ such that $k\coloneqq v(f)<l $.
	 
Now, consider the map 
\[
\phi:	R \to \fb_{l m}/ \fa_{lm}   \quad \text{ sending } \quad g \mapsto  f^m g.
\]
We claim $\ker(\phi) \subset \fa_{ (l-k)m } (v)$. 
Indeed, if $g\in \ker \phi$, then $f^m g \in \fa_{lm}$. Hence, 
\[
v(g) =v(f^m g) - v(f^m) \geq v(\fa_{l m}) - km \geq l m - km = (l-k) m 
\]
as desired. 
Using the claim, we deduce
\begin{equation}\label{e:colengthinequality}
\ell( \fb_{l m}/ \fa_{lm})
\geq
\ell (R/ \ker \phi)
\geq 
\ell ( R/ \fa_{ (l-k)m } (v)) 
.
\end{equation}
Finally, we compute
\[
\e(\fa_\bullet) - \e(\fb_\bullet) = \lim_{  m \to \infty} \frac{ \ell (\fb_{lm } / \fa_{lm})}{(lm)^n /n!}  
\geq 
\lim_{m\to\infty}
\frac{\ell ( R/ \fa_{  (l-k) m }(v) )}{ (lm)^n /n!} = \frac{(l-k)^n \vol(v)}{l^n} >0
,\]
where the first inequality is by \eqref{e:colengthinequality} and the second uses that $v$ has positive volume.
This contradicts our assumption that $\e(\fa_\bullet)=\e(\fb_\bullet)$. 
\end{proof}

\begin{proposition}\label{prop:valuative criterion for multiplicities}
Let $\fa_\bullet$ and $\fb_\bullet$ be  $\fm$-filtrations. Assume $(R,\fm)$ is complete.
If $v(\fb_\bullet) \leq v(\fa_\bullet)$ for all $v\in \DivVal_{R,\fm}$, then $\e(\fb_\bullet) \leq \e(\fa_\bullet)$. 
\end{proposition}
	
The proposition and its proof are local analogues of \cite[Lemma 22]{Sze15}, which concerns  the volumes of graded linear series of projective varieties.

\begin{remark}
The assumption that $R$ is complete in Proposition \ref{prop:valuative criterion for multiplicities}
implies that $R$ is 
 Nagata \cite[\href{https://stacks.math.columbia.edu/tag/032W}{Lemma 032W}]{stacks-project} (see  \cite[\href{https://stacks.math.columbia.edu/tag/033S}{Definition 033S}]{stacks-project}).
 The latter property will be used in the proof of Proposition \ref{prop:valuative criterion for multiplicities} to ensure certain normalization morphisms are proper.
 \end{remark}

\begin{proof}
We claim that there exists a sequence of proper birational morphisms
\[ 
\cdots\to
X_3
\to 
X_2\to X_1\to X_0 = X
\]
such that 
the each $X_i$ is normal and the sheafs $\fa_{l} \cdot \cO_{X_m}$ and $\fb_{l}\cdot \cO_{X_m}$
are line bundles when $l\leq m$.
Such a sequence can be constructed inductively as follows. 
Let $X_{i} \to X_{i-1}$ be defined by the composition
\[
X_{i}  \to X_{i,2}\to X_{i,1} \to X_{i-1},
\]
where $X_{i,1}$ is the blowup of $\fa_i \cdot \cO_{X_{i-1}}$, $X_{i,2}$ is the blowup of $\fb_i \cdot \cO_{X_{i,1}}$, and $X_i$ the normalization of $X_{i,2}$.
Note that $X_{i,2} \to X$ is proper, since $X_{i,2}\to X_{i-1}$ is a blowup and $X_{i-1}\to X$ is proper by our inductive assumption. 
Since $X$ is Nagata and $X_{i,2}\to X$ is finite type, $X_{i,2}$ is Nagata by \cite[\href{https://stacks.math.columbia.edu/tag/035A}{Lemma 035A}]{stacks-project}.
Therefore,  $X_i\to X_{i,2}$ is finite by the definition of Nagata, and we  conclude  the composition $X_i \to X_{i-1}$ is proper, which completes the proof of the claim.

Next, consider  a proper birational morphism $Y\to X$ with $Y$ normal and factoring as
\[
Y \overset{\pi_m}{\longrightarrow} X_m \longrightarrow X,
\]
For $l\leq m$,  there exist Cartier divisors $G_l$ and $G'_l$ on $Y$ such that
\[
\fa_l \cdot \cO_{Y} = \cO_{Y} (G_l)
 \quad \text{ and } \quad 
\fb_l\cdot\cO_{Y}=\cO_{Y}(G_l')
.\]
Set $F_l\coloneqq  \tfrac{G_l}{l}$ and $F'_l\coloneqq  \tfrac{G'_l}{l}$, which are relatively nef over $X$ and satisfy
\[
F_l=- \sum_{E \subset Y} \frac{\ord_E(\fa_l)}{l} E   \quad \text{ and }\quad   F'_l =- \sum_{E \subset Y} \frac{\ord_E(\fb_l)}{l} E,
\]	
where the sums run through prime divisors $E\subset Y$.
Throughout the proof, we will without mention replace $Y$ with higher birational models so  it factors through certain $X_m\to X$.

Given $\epsilon>0$,  by Lemma \ref{lem:intersection formula for mixed mult} and \eqref{eqn:vol=mult}, there exists $m_0>0$ such that
\begin{equation}\label{eqn:approximation for a}
- F_{m_0}^{n-1} \cdot F_{m_0} =\e(\fa_{m_0})/{m_0}^n< \e(\fa_\bullet) + \epsilon. 
\end{equation}
For a multiple $m_1$ of $m_0$, 
\[
  -(F_{m_0})^{n-1} \cdot F_{m_1}< \e(\fa_\bullet) + \epsilon
,\]
 since $F_{m_0} \leq F_{m_1}$ and $F_{m_0}$ is nef over $X$.
Now, we compute
 \begin{align*}
- (F_{m_0})^{n-1} \cdot F_{m_1} 
  +  (F_{m_0})^{n-1} \cdot F'_{m_1}
 &= \sum_{E\subset Y  } \left( - \ord_{E} (F_{m_1}) +\ord_{E} (F'_{m_1}) \right)  (F_{m_0})^{n-1} \cdot E\\
 &= \sum_{E\subset X_{m_0}  } \left( -\ord_{E} (F_{m_1}) +\ord_{E} (F'_{m_1}) \right)  ({\pi_{m_0}}_* F_{m_0})^{n-1} \cdot E,
 \end{align*}
where the second equality  uses that $F_{m_0}$ is the pullback of ${\pi_{m_0}}_* F_{m_0}$ and  the projection formula \cite[Proposition B.16]{Kle05}.
 Since ${\pi_{m_0}}_* F_{m_0}$ is nef over $X$ and
\[
 \lim_{m\to \infty}  \left( -\ord_{E} (F_m) + \ord_{E}(F'_m) \right) = \ord_E(\fa_\bullet)  -\ord_E(\fb_\bullet) \geq 0 
  ,
\]
 we may choose $m_1$ sufficiently large and divisible by $m_0$ such that 
 \[
 -(F_{m_0})^{n-1} \cdot  F'_{m_1}  \leq  - F_{m_0}^{n-1} \cdot  F_{m_1} + \epsilon. 
 \]
Therefore,
 \[
 - (F_{m_0})^{n-2} \cdot F'_{m_1} \cdot F_{m_0} =
- (F_{m_0})^{n-1} \cdot F'_{m_1}   \leq \e(\fa_\bullet) + 2\epsilon
 ,\]
 where the first equality is Proposition \ref{p:symmetric}.

 For any multiple $m_2$ of $m_1$,
 \[
-  (F_{m_0})^{n-2} \cdot F'_{m_1} \cdot F_{m_2}  \leq \e(\fa_\bullet) + 2\epsilon,
 \] 
 since $F_{m_0} \leq F_{m_2}$ and the terms $F'_{m_1}$ and $F_{m_0}$ are nef over $X$.
 Similarly to the previous paragraph, we compute
\begin{multline*}
-  (F_{m_0})^{n-2} \cdot F'_{m_1} \cdot F_{m_2} 
-  (F_{m_0})^{n-2} \cdot F'_{m_1} \cdot F'_{m_2} 
\\
=\sum_{E \subset X_{m_1} } \left( - \ord_{E} (F_{m_2})+ \ord_{E} (F'_{m_2}) \right) ({
 ({\pi_{m_1}}_* F_{m_0})^{n-1} \cdot \pi_{m_1}}_*F'_{m_1})  \cdot E.
\end{multline*}
and, hence, we may choose $m_2$ sufficiently large and divisible by $m_1$ so that 
\[
 -(F_{m_0})^{n-2} \cdot F'_{m_1} \cdot F'_{m_2} <-(F_{m_0})^{n-2} \cdot F'_{m_1} \cdot F_{m_2}+  \epsilon. 
\]
Thus, 
\[
 -(F_{m_0})^{n-2} \cdot F'_{m_1} \cdot F'_{m_2} < \e(\fa_\bullet) +3\epsilon. 
\]
Repeating in this way gives 
\[
-F'_{m_1} \cdot  F'_{m_2} \cdot \, \cdots \, \cdot F'_{m_n} < \e(\fa_\bullet)+ (n+1)\epsilon 
,\] 
where each $m_i$ divides $m_{i+1}$. Since each $F'_{m_i}$ is nef over $X$ and $F'_{m_{i}} \leq F'_{m_{i+1}}$, we see that
\[
-(F'_{m_n})^n < \e(\fa_\bullet)+ n\epsilon. 
\]
Using that $\e(\fb_\bullet) = \inf_{m} \e(\fb_m)/m^n =\inf_{m} -(F'_{m})^n $ by Lemma \ref{lem:intersection formula for mixed mult} and \eqref{eqn:vol=mult},  we see that $\e(\fb_\bullet) < \e(\fa_\bullet) + (n+1)\epsilon $.
Therefore, $\e(\fb_\bullet)\leq \e(\fa_\bullet)$.
\end{proof}

\begin{remark}
When $x\in X$ is an isolated singularity on a normal variety, Proposition \ref{prop:valuative criterion for multiplicities} follows easily from the intersection theory for nef $b$-divisors developed in \cite{BdFF12}. Indeed,  the assumption $v(\fa_\bullet) \leq v(\fb_\bullet)$ implies   $Z(\fa_\bullet)\ge Z(\fb_\bullet)$, where $Z(\fa_\bullet)$ is the nef $b$-Weil $\bR$-divisor associated to $\fa_\bullet$. The proposition then follows from \cite[Remark 4.17]{BdFF12}. 
However, when $x$ is not an isolated singularity, a satisfactory intersection theory for nef $b$-divisors seems  missing from the literature.
\end{remark}


We are now ready to prove Theorem \ref{thm:characterization for sequences with equal volume} and a number of its corollaries. 

\begin{proof}[Proof of Theorem \ref{thm:characterization for sequences with equal volume}]
By Propositions \ref{p:completion-saturation} and \ref{p:completion-multiplicity}, we may assume $(R,\fm)$ is complete. 
This condition will be needed below to apply Proposition \ref{prop:valuative criterion for multiplicities}.

By Propositions \ref{prop:multval} and \ref{prop:valuative criterion for multiplicities}, $\e(\fa_\bullet) = \e(\fb_\bullet)$ if and only if $v(\fa_\bullet) = v(\fb_\bullet)$ for all $v\in \DivVal_{R,\fm}$. 
By Proposition \ref{prop:equiv filts},
 the latter condition holds if and only if $\widetilde{\fa}_\bullet = \widetilde{\fb}_\bullet$.
\end{proof}

\begin{cor}\label{cor:sat mult}
If $\fa_\bullet$ is an $\fm$-filtration, then $\e(\fa_\bullet)=\e(\widetilde{\fa}_\bullet)$.
\end{cor}

\begin{proof}
 By Lemma \ref{lem:Rees is normal},
 $\fa_\bullet\subset \widetilde{\fa}_\bullet$ and $\widetilde{\fa}$ is saturated.
 Thus, Theorem \ref{thm:characterization for sequences with equal volume}  implies $\e(\fa_\bullet) = \e(\widetilde{\fa}_\bullet)$.
\end{proof}

\begin{corollary}\label{cor:equiv equal mult}
Let $\fa_\bullet$ and $\fb_\bullet$ be $\fm$-filtrations. The following statements are equivalent:
\begin{enumerate}
    \item $\widetilde{\fa}_\bullet = \widetilde{\fb}_{ \bullet}$.
    \item $\e(\fa_\bullet) = \e(\fa_\bullet \cap \fb_{\bullet})= \e(\fb_\bullet)$.
    \end{enumerate}
\end{corollary}

	\begin{proof}[Proof of Corollary \ref{cor:equiv equal mult}]
		Assume (1) holds. 
Observe that
		\[
		\e(\fa_\bullet)=\e(\widetilde{\fa}_\bullet) = \e(\widetilde{\fb}_{\bullet})
=  \e(\fb_\bullet)
,\]
where the first and third equality follow from Corollary \ref{cor:sat mult}.
Thus, it remains to show that $\e(\fb_\bullet)= \e(\fa_\bullet \cap \fb_\bullet )$.
First, note that $\e(\fb_\bullet)\leq \e(\fa_\bullet \cap \fb_\bullet )$  holds trivially, since 
  $\fa_\bullet \cap \fb_\bullet \subset \fb_\bullet$.
For the reverse inequality, we compute
		\begin{align*}
			\e(\fa_\bullet\cap \fb_{\bullet}) 
			= 
			\lim_{m \to \infty} \frac{ \ell (R/(\fa_m \cap \fb_{ m}))}{m^n/n!} 
			&= 
			\lim_{m \to \infty} \left( \frac{ \ell (R/(\fa_m) )}{m^n/n!} + \frac{\ell(R/\fb_{ m}) }{m^n/n!}-\frac{ \ell(R/(\fa_m+\fb_{ m}))}{m^n/n!}\right)\\
			&=
			\e(\fa_\bullet) +\e(\fb_{ \bullet}) - 
			\lim_{m\to \infty} \frac{ \ell(R/(\fa_m+\fb_{ m}))}{m^n/n!}
			\\
				&\leq 
			\e(\fa_\bullet) +\e(\fb_{ \bullet}) - 
			\lim_{m\to \infty} \frac{ \ell(R/\widetilde{\fa}_m)}{m^n/n!}
			\\
		&= 
			\e(\fa_\bullet) +\e(\fb_{ \bullet}) - \e(\widetilde{\fa}_\bullet)\\
			& = \e(\fb_\bullet),
		\end{align*}
	where the inequality uses that $\fa_m+\fb_m \subset \widetilde{\fa}_m + \widetilde{\fb}_m = \widetilde{\fa}_m+\widetilde{\fa}_m = \widetilde{\fa}_m$ by the assumption that (1) holds. Therefore, (2) holds. 
		
		Conversely, assume (2) holds.
		Applying Theorem \ref{thm:characterization for sequences with equal volume} to both $(\fa_\bullet \cap \fb_\bullet) \subset \fa_\bullet$ and $(\fa_\bullet \cap \fb_\bullet)\subset \fb_\bullet$, we see that  
		$\widetilde{\fa_\bullet \cap \fb_{ \bullet}}= \widetilde{\fa}_\bullet$ and
		$\widetilde{\fa_\bullet \cap \fb_{ \bullet}}= \widetilde{\fb}_\bullet$, which implies that (1) holds. 
	\end{proof}

The following result was proven when $R$ is regular in \cite[Theorem 1.7.2]{Mus02}.

\begin{corollary}\label{cor:mult>0 = linear growth}
An $\fm$-filtration $\fa_\bullet$ is linearly bounded if and only if $\e(\fa_\bullet)>0$. 
\end{corollary}

\begin{proof}
If $\fa_\bullet$ is linearly bounded, then there exists $c>0$ such that $\fa_{\la} \subset \fm^{\lceil c \la \rceil }$ for all $\la>0$. Thus, $\e(\fa_\bullet) \geq c^{n} \e(\fm) >0$ as desired. 

Next, assume $\e(\fa_\bullet)>0$.
We claim that there exists $v\in \DivVal_{R,\fm}$ such that $v(\fa_\bullet)>0$. 
If not, then $\widetilde{\fa}_\la = \fm$ for all $\la>0$. Using Corollary \ref{cor:sat mult}, we then see
$\e(\fa_\bullet ) = \e(\widetilde{\fa}_\bullet) =0$
, which is a contradiction. 
Now, fix $v\in \DivVal_{R,\fm}$ with $v(\fa_\bullet) >0$. 
Using that  $\fa_\bullet \subset \fa_{\bullet}(v)$ and Proposition \ref{prop:qm>0vol}, we conclude 
 $\e(\fa_\bullet)\geq \e(\fa_\bullet(v)) >0$.
\end{proof}

\subsection{Saturation and finite-volume valuations}
Using results from the previous section, we show that the saturation can be defined using positive volume valuations, rather than only divisorial valuations.

\begin{prop}\label{prop:saturation alt def}
If $\fa_\bullet$ is an $\fm$-filtration and $\la\in \mathbb{R}_{>0}$, then 
\[
\widetilde{\fa}_\la = \{ f\in \fm \, \vert \,   v(f) \geq  \la \cdot v(\fa_\bullet) \text{ for all } v\in \Val_{R,\fm}^+ \} 
\]
\end{prop}

The  proposition will be deduced from the following lemma.

\begin{lemma}\label{lem:valuative ideals sat}
If $\{ v_i \}_{i\in I}$ is a collection of valuations in $\Val_{R,\fm}^+$ and $\{ c_i \}_{i\in I}$ of non-negative real numbers, then the $\fm$-filtration $\fa_\bullet$ defined by
\[
\fa_\la \coloneqq \{ f\in \fm \, \vert \, v_i (f) \geq \la c_i  \text{ for all } i \in I\}
\] 
is saturated. Hence,  $\fa_\bullet(v)$ is saturated for any $v\in \Val_{R,\fm}^+$.
\end{lemma}

\begin{proof}
Suppose the statement is false. Then there exists some $\lambda \in \R_{>0}$ such that $\fa_\lambda \subsetneq \widetilde{\fa}_{\lambda}$. 
Thus, there exist $f\in \widetilde{\fa}_{\lambda}$ and $i \in I$ such that $v_i(f) < \lambda c_i$.  
We claim  $v_i(\widetilde{\fa}_\bullet)< v_i(\fa_\bullet)$. 
Indeed,
\[
v_i(\widetilde{\fa}_\bullet)=
\lim_{m \to \infty} \frac{ v_i(\widetilde{\fa}_{ \lfloor \lambda m\rfloor  })}{  \lfloor \la m \rfloor }
\leq \lim_{m\to \infty}  \frac{ v_i(f^m)}{ \lfloor \lambda m \rfloor } 
= 
\frac{ v_i(f) }{\lambda}<c_i,
\]
where the second equality uses that $f^m \in (\widetilde{\fa}_{\lambda})^m \subset \widetilde{\fa}_{\la m} \subset \widetilde{\fa}_{\lfloor \la m \rfloor}$. 
Since $v_i(\fa_\bullet) \geq c_i$, the claim holds.

By Proposition \ref{prop:multval} and the claim,  $\e(\widetilde{\fa}_\bullet)< \e(\fa_\bullet)$. The latter contradicts Theorem \ref{thm:characterization for sequences with equal volume}. 
\end{proof}

\begin{proof}[Proof of Proposition \ref{prop:saturation alt def}]
Let $\fb_\bullet$ denote the $\fm$-filtration defined by 
\[
\fb_\la \coloneqq  \{f \in \R \, \vert \, v(f) \geq \la \cdot v(\fa_\bullet) \text{ for all } v \in \Val_{R,\fm}^+ \}
.\]
Notice that $\fa_\bullet \subset \fb_\bullet \subset \widetilde{\fa}_\bullet$, 
where the second inclusion uses that all divisorial valuations have positive volume. 
Taking saturations gives
$\widetilde{\fa}_\bullet \subset \widetilde{\fb}_\bullet \subset \widetilde{\widetilde{\fa}}_\bullet$.
Since $\fb_\bullet$ and $\widetilde{\fa}_\bullet$ are saturated by Lemma \ref{lem:valuative ideals sat}, we conclude
$\widetilde{\fa}_\bullet \subset \fb_\bullet \subset {\widetilde{\fa}}_\bullet$.
\end{proof}

\section{Multiplicity and geodesics}\label{sec:multandgeodesics}
In this section, we prove Theorem \ref{thm:main theorem} on the convexity of the multiplicity function along geodesics. 
Throughout, $(R,\fm,\kappa)$ denotes an $n$-dimensional, analytically irreducible, Noetherian  local domain  containing a field of arbitrary characteristic.

\subsection{Geodesics}
Fix  two finite-multiplicity $\fm$-filtrations $\fa_{\bullet,0}$ and $\fa_{\bullet,1}$.

\begin{definition}\label{d:geodesic}
For each $t\in (0,1)$, we define an $\fm$-filtration
$\fa_{\bullet,t}$  by setting 
\[
\fa_{\lambda,t}  = \sum_{\lambda=(1-t)\mu+t\nu} \fa_{\mu,0} \cap \fa_{\nu,1}
,\]
where the sum runs through all $\mu, \nu \in \bR$ satisfying $\lambda=(1-t)\mu+t\nu$.
We call $(\fa_{\bullet,t})_{t\in [0,1]}$ the \emph{geodesic} between $\fa_{\bullet,0}$ and $\fa_{\bullet,1}$.

This definition is a local analog of the geodesic between two  filtrations of the section ring of a polarized variety  \cite{BLXZ21,Reb20}. See Section {sec:global} for details.

\end{definition}

\begin{example}
Let $R\coloneqq  k[x,y]_{(x,y)}$. For $\bm{\alpha}=(\bm{\alpha}_1,\bm{\alpha}_2) \in \R_{>0}^2$, let $v_{\bm{\alpha}}\colon {\rm Frac}(R)^\times \to  \R$ be the monomial valuation with weight $\bm{\alpha}_1$ and $\bm{\alpha}_2$ with respect to $x$ and $y$, that is
\[
v\big( \sum_{m,n \geq 0 } c_{m,n} x^{m}y^{n} \big) = \min \{ \bm{\alpha}_1m + \bm{\alpha}_2 n  \, \vert \, c_{m,n} \neq 0\}.\]
Note that
$\fa_\la (v_{\bm{\alpha}}) = \{x^{m}y^n \, \vert \, m \alpha_1 + n \alpha_2 \geq \la\}$.
Now, fix $\bm{\alpha}, \bm{\beta} \in \bR_{\geq0}^2$ and consider the geodesic $\fa_{\bullet,t}$ between
 $\fa_{\bullet,0} \coloneqq  \fa_{\bullet}(v_{\bm{\alpha}})$
 and $\fa_{\bullet,1}\coloneqq  \fa_{\bullet}(v_{\bm{\beta}}) $.
 A short computation shows 
\[
\fa_{\bullet,t} = \fa_\bullet( v_{ (1-t)\bm{\alpha} + t \bm{\beta}})
.\]
This computation unfortunately does not generalize to the case of quasi-monomial valuations.
See \cite[Section 4]{LXZ21} for a study of this failure in the global setting.
\end{example}

To prove Theorem \ref{thm:main theorem}, we  define a measure on $\R^2$ that  encodes the multiplicities along the geodesic. 
The argument may be viewed as a local analogue of the construction in \cite[Section 3.1]{BLXZ21} for the geodesic between two filtrations of the section ring of a polarized variety. 
The latter global construction is motivated by \cite{BC11} and \cite[Thmeorem 4.3]{BHJ17}, which constructs a measure on $\bR$ associated to a single filtration of the section ring of polarized variety.

 Before proceeding with the proof,  fix integers $C>0$ and $D>0$ such that
 \begin{equation*}
 \begin{gathered}
 \fm^{ \lceil C \la \rceil } \subset \fa_{\la,0} \cap \fa_{\la,1},\\
  \fa_{D \la,0 } \subset \fa_{\la, 1} \quad\quad \quad \text{ and } \quad\quad \quad  \fa_{D\la,1} \subset \fa_{\la,0}.
  \end{gathered}
\end{equation*}
 for all $\la \geq 1$. 
 The existence of $C$ follows from Lemma \ref{lem:linearlybounded},
 while the existence of $D$ follows from Lemma \ref{lem:linearlybounded} and the fact that $\fa_{\bullet,0}$ and $\fa_{\bullet,1}$ are both linearly bounded by Corollary \ref{cor:mult>0 = linear growth}.

\subsection{Sequences of measures}
For each  $m\in \bZ_{>0}$,   consider the   function $H_m\colon \bR^2 \to \bR$ defined by   
\[
H_m(x,y) 
\coloneqq
\frac{\ell \left( R/  (\fa_{mx,0} \cap \fa_{my,1}) \right)}
{ m^n/n! }
.
\]
Notice that $H_m$ is non-decreasing and left continuous in each variable. Using the sequence $(H_m)_{m\geq1}$, we define a sequence of measures on $\R^2$.

\begin{prop}\label{p:measuremum}
The distributional derivative $\mu_m\coloneqq - \frac{ \partial^2 H_m}{\partial x \partial y}$ is a discrete  measure on $\bR^2$ and has support contained in $ \{  \tfrac{1}{D} x \leq y \leq Dx \} \cup \{ 0 \leq x \leq \tfrac{1}{m} \}  \cup \{ 0 \leq y\leq \tfrac{1}{m} \}$.
\end{prop}

Before proving the proposition, we prove the following lemma allowing us to reduce to the case of complete local rings. 

\begin{lem}\label{l:geodesiccomplete}
Let $(\widehat{R},\widehat{\fm})$ denote the completion of $(R,\fm)$, 
$\fb_{\bullet,i}= \fa_{\bullet,i} \cdot \widehat{R}$ for $i=0,1$, 
and
 $(\fb_{t,\bullet})_{t\in [0,1]}$ the geodesic between $\fb_{\bullet,0}$ and $\fb_{\bullet,1}$.
 Then the following statements hold:
 \begin{enumerate}
 \item $\ell ( {R}/ \fa_{a, 0} \cap \fa_{b,1})= \ell( \hat{R}/ \fb_{a, 0} \cap \fb_{b,1}) $ for each $a,b\geq0$.
 \item $\ell( {R}/ \fa_{a,t}) = \ell( \hat{R}/ \fb_{a, t} ) $ for each $a\geq 0$ and $t\in [0,1]$.
 \end{enumerate}
 \end{lem}

\begin{proof}
For $\fm$-primary ideals $\fc\subset R$ and $\fd\subset R$, 
\[
(\fc + \fd )\widehat{R} = \fc \widehat{R} + \fd \widehat{R}
\quad \text{  and } \quad
(\fc \cap \fd )\widehat{R} = \fc \widehat{R} \cap \fd \widehat{R}
.\]
Therefore, 
\[
(\fa_{a, 0} \cap \fa_{b,1} ) \widehat{R} = \fb_{a, 0} \cap \fb_{b,1}
\quad \text{ and } \quad
\fa_{a,t}  \widehat{R} = \fb_{a, t}
.\]
Statements (1) and (2) follow from these equalities.
\end{proof}

\begin{proof}[Proof of Proposition \ref{p:measuremum}]
By Lemma \ref{l:geodesiccomplete}, it suffices to prove the statement when $(R,\fm)$ is complete. 
Since $R$ contains a field by assumption and is complete, there is an inclusion  $\kappa \hookrightarrow  R$ such that the composition $\kappa \hookrightarrow R \to \kappa$ is the identity. This will be helpful, since any finite-length $R$-module is naturally a finite-dimensional $\kappa$-vector space via restriction of scalars.

Fix an integer $N>0$ and consider the finite-dimensional $\kappa$-vector space 
$V\coloneqq R/ \fm^{NCm}$.
The $\fm$-filtrations $\fa_{\bullet,0} $ and $\fa_{\bullet,1}$ induce decreasing filtrations 
$\cF^\bullet _0 $ and $\cF^\bullet_1 $ of $V$ defined by
\[
\cF_0^\lambda  V 
\coloneqq 
{\rm im}( \fa_{\la,0} \to V) 
\quad \quad
 \text{ and } 
 \quad \quad  
\cF_1^\lambda V
\coloneqq
{\rm im}(\fa_{\la,1} \to V)
.\]
Observe that when $x<N$ and $y<N$,
\[
\ell \left( R /  (\fa_{mx,0} \cap \fa_{my,1}) \right) 
=
 \dim_\kappa \big( V / (\cF_0^{mx,0} V \cap \cF_1^{my,1} V )\big)
=
 \dim_\kappa V- \dim_\kappa (\cF_0^{mx,0} V \cap \cF_1^{my,1} V),
\]
where the first equality uses that $\fm^{NCm} \subset \fa_{mx,0} \cap \fa_{my,1}$.

To analyze the dimension of $\cF_0^{mx} V \cap \cF_1^{my} V$, we use that any finite-dimensional vector space with two filtrations  admits a basis simultaneously diagonalizing both filtrations (see \cite[Proposition 1.14]{BE18} and \cite[Lemma 3.1]{AZ20}).
This means there exists a basis $(s_1,\ldots, s_\ell)$ for $V$ such that each 
$\cF_j^\lambda V$ is the span of some subset of the basis elements. Hence, if we set 
$\lambda_{i,j} \coloneqq  \sup \{ \lambda \in \bR\, \vert\, s_i \in \cF_j^\lambda V \}$,
then 
\[
\cF_0^\lambda V ={\rm span} \langle s_i \, \vert \, \lambda_{i,0} \geq\lambda \rangle
\quad \text{ and } \quad 
\cF_1^\lambda V ={\rm span} \langle s_i \, \vert \, \lambda_{i,1} \geq\lambda \rangle
.\]
Using that  
$\dim_\kappa \cF_0^{mx} V \cap \cF_1^{my} V 
= 
\# \{ i \, \vert \, \lambda_{i,0} \geq mx  \text{ and }\lambda_{i,1} \geq my \}$,
 we compute
\[
 \frac{ \partial^2}{\partial x \partial y }  \dim_\kappa (\cF_0^{mx} V \cap \cF_1^{my} V) 
 = 
 \sum_{i=1}^{\ell} \delta_{ (m^{-1} \lambda_{i,0}, m^{-1}  \lambda_{i, 1} )}
 .
 \]
 Therefore, the restriction of $\mu_m$ to $(-\infty, N)\times (-\infty, N)$ equals
\begin{equation}\label{e:mumrestrict}
\mu_m \big\vert_{ (-\infty,N)\times (-\infty,N)}
=
\frac{1}{m^n/n!}
\sum_{  i \, \vert \, \lambda_{i,0} , \lambda_{i,1} < mN} \delta_{  (m^{-1} \lambda_{i,0}, m^{-1} \lambda_{i,1} )}.
\end{equation}
 Since $N>0$ was arbitrary, $\mu_m$ is a discrete measure on $\bR^2$.

It remains to analyze the support of $\mu_m$.
Note that $\cF_0^{\la D} V \subset \cF_1^{\la} V$ and $\cF_1^{\la D} V \subset \cF_0^\la V$  for all $\la\geq 1$. 
Therefore,
\begin{itemize}
\item[(i)] If $\la_{i,1} >1$, then $\la_{i,0} \leq D \la_{i,1}$. 
\item[(ii)] If $\la_{i,0} >1$, then $\la_{i,1} \leq D \la_{i,0}$. 
\end{itemize}
Using (i), (ii), and \eqref{e:mumrestrict}, we see that $\supp (\mu_m)  \subset  \{  \tfrac{1}{D} x \leq y \leq Dx \} \cup \{ 0 \leq x \leq \tfrac{1}{m} \}  \cup \{ 0 \leq y\leq \tfrac{1}{m} \} $.
\end{proof}

Using  \eqref{e:mumrestrict}, we show that $\mu_m$ encodes the colengths of $\fa_{\bullet,t}$.

\begin{prop}\label{p:col}
For each $t\in [0,1]$ and $a,b\in \bR_{\geq0}$,
\begin{align*}
\mu_m\left( \{   (1-t)x+ty < a \} \right)& = \frac{\ell \left( R/ \fa_{ma,t}\right)}{m^n/n!}  \\
 \mu_m\left( \{  x<a\} \cup \{ y<b\} \right) & = \frac{ \ell \left(  R/ ( \fa_{ma,0} \cap \fa_{mb,1})\right)}{m^n/n!} 
\end{align*}
\end{prop}

\begin{proof}
By Lemma \ref{l:geodesiccomplete}, it suffices to prove the result when $(R,\fm)$ is complete. 
In this case we continue with the notation from the previous proof, 
but additionally fix $N>\max  \{ \tfrac{a}{1-t}, \tfrac{b}{t}\}$.
Now, consider the filtration $\cF^\bullet_t$ of $V$ defined by 
$\cF_t^\lambda V\coloneqq {\rm im} ( \fa_{\la,t} \to V) $ for each $ \lambda \in \bR$. 
Observe that
\begin{equation}\label{e:FtaV}
\cF_t^{ma} V 
= 
\sum_{ ma = (1-t) \mu + t \nu } \cF_0^\mu V \cap \cF_1^\nu V 
=
{\rm span} \langle  s_i \, \vert\, (1-t)\lambda_{i,0}  +t \lambda_{i,1} \geq ma \rangle,
\end{equation}
where the sum runs through all $\mu, \nu\in \bR$ satisfying $ma=(1-t)\mu+t \nu$.
Now,  we compute 
\begin{multline}
\ell \left( R/\fa_{ma,t} \right)
=
\dim_\kappa V / \cF_t^{ma} V 
 = 
 \dim_\kappa V- \dim_\kappa \cF_t^{ma} V \\
=  
\# \{ i \, \vert \, (1-t) \lambda_{i,0}  +  t\lambda_{i,1} <  ma \} 
= 
\frac{m^n}{n!}\mu_m\left( \{  (1-t)x + ty < a\}  \right) 
\end{multline}
where the first equality uses that $\fm^{NCm} \subseteq \fa_{ma/(1-t),0} + \fa_{mb/t,1} \subset \fa_{ma,t}$, the third \eqref{e:FtaV}, and the fourth \eqref{e:mumrestrict}. 
Therefore, the first formula in  the proposition holds.

To verify the second formula, we fix $N>  \max \{a,b\}$ and similarly compute 
\begin{multline*}
\dim_\kappa R/( \fa_{ma,0} \cap \fa_{mb,1} )  
= 
\dim_\kappa V /  (\cF_0^{ma} V  \cap \cF_1^{mb} V) 
= 
\dim_\kappa V- \dim_\kappa \cF_0^{ma} V \cap \cF_1^{mb} V
\\
= 
 \# \{ j \, \vert \, \lambda_{i,0} < ma   \text{ or }  \lambda_{i,1} <mb  \}  
= 
\frac{m^n}{n!}\mu_m\left( \{ x<a \} \cup \{y <b \} \right),
\end{multline*}
where the first equality uses that $\fm^{NCm} \subset \fa_{ma,0} \cap \fa_{mb,1}$.
\end{proof}

\subsection{Limit measure}
We  now  construct a limit of the  sequence of measures $(\mu_m)_{m \geq1}$ that encodes the multiplicity along the geodesic.

Consider 
the function $H\colon \bR^2 \to \bR$ defined by
\[
H(x,y) =   \e( \fa_{x\bullet ,0} \cap \fa_{y\bullet ,1} )
.\]
Above, we use the convention that $\e(R)=0$, which  occurs when  $x\leq0 $ and $y\leq0$.


%

\begin{prop}\label{p:measuremu}
The distributional derivative $\mu\coloneqq - \frac{ \partial^2 H}{\partial x \partial y}$ is a  measure on $\bR^2$
and the sequence of measures $(\mu_m)_{m\geq 1}$ converges weakly to $\mu$ as $m\to \infty$. 
\end{prop}

\begin{proof}
We claim that $H_m$ converges to $H $ in $L^1_{ {\rm loc}}( \bR^2)$. Assuming the claim,  then $H_m$ converges to $H$  as distributions, and, hence, $\mu_m $ converges to $\mu\coloneqq - \frac{ \partial^2 H}{\partial x \partial y}$ as distributions as well. 
Since each $\mu_m$ is a measure, \cite[Theorem 2.1.9]{Hor03} implies 
$\mu$ is a measure and $\mu_m$ converges to $\mu$ weakly as measures. 

 It remains to prove the above claim. First, observe that $H_m$ converges to $H$ pointwise by the definition of the multiplicity as a limit. Next, fix an integer $N>0$. For $-N\leq x,y<N$, observe that 
 \[
 0 \leq H_m(x,y) 
 \leq 
   \frac{ \ell \left( R/ \fm^{N Cm} \right) }{ m^n/n! }  
    \leq 
  \sup_{m \geq 1}  \frac{\ell \left( R/ \fm^{NCm}\right) }{ m^n/n! }  
  <+\infty
 .  \]
 Therefore, the dominated convergence theorem implies that
  \[
 \lim_{m \to \infty} \int_{ (-N,N)\times (-N,N) } |H_m-H| =0. 
 \]
 Since  $N>0$ was arbitrary, $H_m$ converges to $H$ in $L^1_{ {\rm loc}}( \bR^2)$ as desired. 
\end{proof}

\begin{prop}\label{p:mult}
For each $t\in [0,1]$ and $a,b\in \bR_{\geq0}$,
\begin{itemize}
    \item[(1)] $a^n \e(  \fa_{\bullet ,t})=\e(\fa_{a\bullet,t}) =\mu( \{ x(1-t)+yt < a \})$
    \item[(2)] $ \e( \fa_{a\bullet,0 } \cap \fa_{b\bullet,1})= \mu \left( \{ x<a\} \cup \{y<b\}\right)$.
\end{itemize}
\end{prop}

\begin{proof}
Fix $\epsilon>0$. Since $(\mu_{m})_{m\geq1}$ converges to $\mu$ weakly as measures by Proposition \ref{p:measuremu}, 
\begin{multline*}
\limsup \mu_m\left( \{ x(1-t) +y t< a-\epsilon \} \right)
\leq 
\mu\left( \{ x(1-t) +y t< a\}\right)\\ \leq \liminf \mu_m \left( \{x(1-t)+yt < a\} \right)
\end{multline*}
Using Proposition \ref{p:col}(1) to compute the lim sup and lim inf, we deduce
\[
(a-\epsilon)^n\e(\fa_{\bullet,t})
=
\e(\fa_{(a-\epsilon)\bullet,t})
\leq \mu\left( \{ x(1-t)+yt<a\}\right) \leq   \e(\fa_{a\bullet,t})= a^n \e(\fa_{\bullet,t}). 
\]
Sending $\epsilon\to0$ completes the proof of (1). The proof of the second is similar, but uses Proposition \ref{p:col}.2.
\end{proof}

 


\begin{lem}\label{l:suppmu}
For any real number $c>0$, the following statements are equivalent:
\begin{enumerate}
\item   $\supp(\mu) \subset \{ cx =y \}$.
\item  $\e(\fa_{\bullet,0}) = \e(\fa_{\bullet,0} \cap \fa_{c\bullet,1}) = \e(\fa_{c\bullet,1})$.
\end{enumerate}
\end{lem}

\begin{proof}
If (1) holds, then 
\[
\mu\left(\{ x< 1\} \right) = \mu \left(\{ x <1\} \cup \{  y <c \} \right) = \mu\left( \{y< c \}\right).
\]
Proposition \ref{p:mult} then implies (2) holds. 

For the reverse implication,  assume (2) holds. We claim  
\begin{equation*}
H(x,y)= 
\begin{cases}\label{e:Hxy}
\e(\fa_{x\bullet,0} )  \quad & \text{ if } cx\geq y  \\
\e(\fa_{ y\bullet,1} )  \quad & \text{ if } y\geq cx 
\end{cases}
.\end{equation*}
Indeed, if $cx\geq y$, then 
\[
\e(\fa_{x\bullet ,0} ) \leq H(x,y)
\leq
\e(\fa_{x \bullet,0 } \cap \fa_{cx\bullet ,1})
= 
 \e(\fa_{ x\bullet ,0})
 .
\]
Thus, the claim holds when $cx\geq y$ and similar reasoning treats the case when $cx< y$. 
Using the claim, we conclude (1) holds.
\end{proof}

\begin{prop}\label{p:muhomog}
The following statements hold:
\begin{enumerate}
\item The support of $\mu$ is contained in $\{\tfrac{1}{D} x \leq y \leq  D x\}$. 
\item $\mu$ is homogeneous of degree $n$ (i.e. $\mu(c\cdot A) = c^n \mu ( A)$ for any Borel set $A\subset \bR^2$ and real number $c>0$).
\end{enumerate}
\end{prop}

\begin{proof}
By Propositions  \ref{p:measuremum} and \ref{p:measuremu}, $\supp(\mu) \subset 
\{\tfrac{1}{D} x \leq y \leq  D x\} \cup \{ x=0\} \cup \{y =0 \}$. 
Note that 
\[
\mu( \{x =0 \})
= 
\lim_{ a \to 0^+} \mu (\{  0\leq x< a \} )
= 
\lim_{a \to 0^+} 
\e( \fa_{a \bullet,0}) 
=
\lim_{a \to 0^+} 
a^n \e( \fa_{ \bullet,0})
=
0,
\]
where the second equality is by Proposition \ref{p:mult}.
Similar reasoning shows $\mu(\{ y=0 \})$. 
Thus, (1) holds.
Statement (2), follows from the fact that $H$ is homogeneous of degree $n$, meaning $H(cx,cy) = c^n H(x,y)$ for each $c>0$ and $(x,y) \in \bR^2$.
\end{proof}

\subsection{Measure on the segment}
To prove Theorem \ref{thm:main theorem}, it will be convenient to work with a measure induced by $\mu$  on the interval $[0,1]$.

Consider the embedding 
\[
j\colon [0,1] \hookrightarrow \R^2 \quad \text{ defined by  }\quad  j(z) \coloneqq (z,1-z).\]
Note that $j([0,1])$ is the line segment between $(0,1)$ and $(1,0)$ in $\bR^2$.
For  $A\subset [0,1]$,  set
\[
{\rm Cone}(A) \coloneqq \R_{\geq0} \cdot j(A)
.\]
We define a  measure $\widetilde{\mu}$ on $[0,1]$ by setting
\[
\widetilde{\mu}( A) \coloneqq \mu \left(  {\rm Cone}(A) \cap \{ x+ y < 1 \} \right)
\]
for any Borel set $A \subset [0,1]$.
Note that  $\widetilde{\mu}$ is indeed a measure, since $\mu$ is a measure by Proposition \ref{p:measuremu} and $\mu( \{{\bf 0}\} )=0$, which follows from Proposition \ref{p:mult}.
Additionally, $\supp (\widetilde{\mu})\subset \big[\tfrac{1}{D+1}, \tfrac{D}{D+1} \big]$, since $\supp(\mu)\subset  \{ \tfrac{1}{D} x \leq y \leq Dx \}$. 

\begin{lem}\label{l:supptildemu}
For any real number $c>0$, the following statements are equivalent: 
\begin{enumerate}
\item $\supp(\tilde{\mu}) = \tfrac{1}{c+1}$.
\item $\e(\fa_{\bullet, 0}) = \e(\fa_{\bullet,0} \cap \fa_{c\bullet,1}) = \e(\fa_{c\bullet, 1})$.
\end{enumerate}
\end{lem}

\begin{proof}
Using that $\R\cdot j \big( \tfrac{1}{c+1}\big)= \{cx=y \}$  and the definition of $\widetilde{\mu}$, we see that
 (1) holds if and only if   ${\rm supp}(\mu\vert_{\{0\leq x+y<1\}}) \subset \{cx=y \}$.
 Since $\mu$ is homogeneous by Proposition \ref{p:muhomog}(2), 
 ${\rm supp} \left( \mu\vert_{\{0\leq x+y<1\}} \right) \subset \{xc=y\}$
 if and only if ${\rm supp}(\mu) \subset \{cx=y\}$.
 By Proposition \ref{l:suppmu}, the latter condition holds if and only if (2) holds.
 \end{proof}

\begin{lem}\label{l:Vformula}
For $t\in [0,1]$,
\[ \e(\fa_{\bullet,t} ) =  \int \left( (1-t)z +t(1-z) \right)^{-n} \, {\rm d} \widetilde{\mu}
.\]
\end{lem} 

\begin{proof}
For simplicity, we assume $ t \leq \tfrac{1}{2}$, which ensures $g_t(z): =(1-t)z +t(1-z)$ is non-decreasing. (The case when $t>\tfrac{1}{2}$ can be proved similarly, but using left-hand approximations below.)
For each $m\geq 1$, consider the simple function
\[
g_{t,m}(z) = \sum_{k=1}^{2^m}  g_{t}(k/2^m)   \cdot 1_{I_{k,m} } 
.\]
where $I_{1,m} \coloneqq  [0,\tfrac{1}{2^m}]$ and $I_{k,m}\coloneqq  \left( \tfrac{k-1}{2^m}, \tfrac{k}{2^m} \right]$ when $1<k\leq 2^m$ . Since $g_{t,m} \geq g_{t,m+1}$ and $g_{t,m}\to g$ pointwise as $m\to \infty$, the monotone convergence theorem implies 
\[
\int g_t(z)^{-n} \, {\rm d} \widetilde{\mu} = \lim_{m \to \infty}
\int g_{t,m}(z)^{-n} \, {\rm d} \widetilde{\mu}. 
\]
(Above we are using that  $g_t(z)^{-n}$ is defined and continuous on $(z,t) \in  \supp (\widetilde{\mu}) \times [0,1]$).)
Now, set $A_{k,m} \coloneqq  g_{t}(k/2^m)^{-1}\cdot  \left(  {\rm Cone}(I_{k,m}) \cap \{x+y < 1\} \right) \subset \bR^2$ and $A_m = \cup_{k=1}^{2^m} A_{k,m}$.
 We compute  
\[
 \int g_{t,m}(z)^{-n} \, {\rm d} \widetilde{\mu} 
 = 
 \sum_{k=1}^{2^m} g_{t}(k/2^m)^{-n} \widetilde{\mu}(I_{k,m}) 
= 
\sum_{k=1}^{2^m} \mu\left( A_{k,m}  \right)
= 
\mu(A_m)
,\]
where the second equality uses Proposition \ref{p:muhomog}.2.
Next, notice that $A_{m} \subset A_{m+1}$ and 
\begin{multline*}
\bigcup_{m \geq 1}
A_m = \bigcup_{m \geq 1} \bigcup_{z\in [0,1] }g_{t,m}(z)^{-1} \cdot ({\rm Cone}(\{z\}) \cap \{ x+y <1 \}) \\
= 
\bigcup_{z\in [0,1]} g_{t}(z)^{-1} \cdot ({\rm Cone}(\{z\}) \cap \{x+y <1 \})
= \mathbb{R}^2_{>0}  \cap \{(1-t) x+ty <1 \}.
\end{multline*}
Combining the previous two computations, we conclude 
\[
\int g_t(z)^{-n} \, {\rm d} \widetilde{\mu}=
\lim_{m\to \infty} \mu(A_m)
= 
\mu \left(  \mathbb{R}^2_{>0} \cap \{ (1-t)x+ ty <1 \} \right)
= 
\e(\fa_{\bullet,t})
,\]
where the final equality is Proposition \ref{p:mult}(1).
\end{proof}

\subsection{Proof of Theorem \ref{thm:main theorem}}

\begin{proof}[Proof of Theorem \ref{thm:main theorem}]
Set $g(z,t) \coloneqq  (1-t)z + t(1-z)$. 
By Lemma \ref{l:Vformula},
\[
E(t) =  \int g(z,t)^{-n} \, {\rm d}\widetilde{\mu}
.\]
Since
$g(z,t) = 0$  if and only if  $z= \frac{t}{2t-1}$,
there exists $\epsilon>0$ such that 
\[
g(z,t) \neq  0 \text{ for all } (z,t) \in    \big [  \tfrac{1}{D+1}-\epsilon,\tfrac{D}{D+1} +\epsilon\big] \times [-\epsilon, 1+\epsilon]
.\]
Since $g(z,t) $ is also smooth on $\bR^2$, $g(z,t)^{-n}$ is smooth and has bounded derivatives on $ (\tfrac{1}{D+1}-\epsilon,\tfrac{D}{D+1}+\epsilon) \times (-\epsilon, 1+\epsilon)$.
Using that ${\rm \supp} {\tilde{\mu} } \subset \big [ \tfrac{1}{D+1},\tfrac{D}{D+1} \big]$,  the Leibniz integral formula implies  $\int g(z,t)^{-n} {\rm d} \widetilde{\mu}$ is smooth on $(-\epsilon,1+\epsilon)$, which  proves (1).  
The  formula additionally implies
\[
E'(t)= n \int (2z-1)g(z,t)^{-n-1}\, {\rm d} \tilde{\mu} \,\,\, \text{ and }   \,\,\,
E''(t) = n(n+1) \int (2z-1)^2g(z,t)^{-n-2} \, {\rm d} \tilde{\mu}.
\]

To prove (2), it suffices to show $\frac{d^2}{dt^2} \left(E(t)^{-1/n} \right)\leq 0$ for all $t\in (0,1)$. 
Note that 
\[
\frac{d^2}{dt^2} \left(E(t)^{-1/n} \right)
=
\frac{(n+1) E'(t)^2- n E(t)E''(t) }{ n^2E(t)^{1/n\,+\,2}  }
.\]
To show the latter quantity is non-positive, we compute
\begin{multline}\label{e:Vcomp}
nE(t) E''(t) = n^2(n+1)
 \left( \int \left| g(z,t)^{-n} \right| \, {\rm d} \tilde{\mu} \right)
\left( \int  \left| (2z-1)^2{g(z,t)^{-n-2}} \right| \, {\rm d} \tilde{\mu} \right) \\
\geq 
n^2 (n+1)
\left(
\int \left| (2z-1) g(z,t)^{-n-1} \right| \, {\rm d} \tilde{\mu}\right)^2
\geq 
(n+1) E'(t)^2
,\end{multline}
where the first inequality is by the Cauchy-Schwarz inequality. 
Therefore,  $\frac{d^2}{dt^2} \left(E(t)^{-1/n} \right) \leq0$ for all $t\in (0,1)$, which proves (2).

To prove (3), note that $E(t)$ is linear 
 if and only if 
 $\frac{d^2}{dt^2} \left(E(t)^{-1/n} \right)=0$ for all $t\in (0,1)$.
 The latter holds if and only if the inequalities in \eqref{e:Vcomp} are all  equalities. 
 The inequalities are equalities iff 
  $(2z-1)g(z,t)^{-(n+2)/2}$ and  $g(z,t)^{-n/2}$ are linearly dependent in $L^{1}(\widetilde{\mu})$.
 Equivalently, $2z-1$ and $g(z,t)$ are linearly dependent in $L^{1} (\widetilde{\mu})$. 
The linear dependence holds exactly when $\widetilde{\mu}$ is supported at a single point. 
By Lemma \ref{l:supptildemu} and Corollary \ref{cor:equiv equal mult}, the  latter  condition is equivalent to the existence of  $c\in \bR$ such that  $\widetilde{\fa}_{\bullet, 0}=\widetilde{\fa}_{c\bullet,1}$. 
\end{proof}

\section{Applications}\label{sec:proof}
In this section, we prove Corollaries  \ref{cor:convexityofvolume}, \ref{cor:uniqueness of minimizer}, and \ref{cor:Minkowski} using results from previous sections.

\subsection{Convexity of the volume of a valuation}

\begin{proof}[Proof of Corollary \ref{cor:convexityofvolume}]
Let $(\fa_{\bullet,t})_{t\in [0,1]}$ denote the geodesic between $\fa_{\bullet,0}\coloneqq \fa_\bullet(v_{\bm{\alpha}})$ and $\fa_{\bullet,1}\coloneqq \fa_{\bullet}(v_{\bm{\beta}})$.
Since $v_{(1-t) \bm{\alpha} + t \bm{\beta}}(f) \geq (1-t) v_{\bm{\alpha}} (f) + tv_{\bm{\beta}} (f)$ for any $f\in \cO_{X,x}$, 
$\fa_{\bullet,t}\subset \fa_\bullet(v_{(1-t)\bm{\alpha} + t \bm{\beta}})$.
We compute 
\begin{multline*}
\vol(v_{(1-t)\bm{\alpha} + t \bm{\beta}})^{-1/n}\coloneqq \e(\fa_\bullet((v_{(1-t)\bm{\alpha} + t \bm{\beta}}))^{-1/n}
            \geq \e(\fa_{\bullet,t})^{-1/n}\\
			\geq (1-t)\e(\fa_{\bullet,0})^{-1/n}+ t \e(\fa_{\bullet,1})^{-1/n}
			=(1-t)\vol(v_0)^{-1/n}+ t \vol(v_1)^{-1/n},
\end{multline*}
where the second inequality is Theorem \ref{thm:main theorem}(2). 
Additionally, by Theorem \ref{thm:main theorem}(3) and  Lemma \ref{lem:valuative ideals sat}, the second inequality is strict unless there exists $c\in \bR_{>0}$ such that $\fa_{\bullet,0}= \fa_{c\bullet,1}$. 
Since $\fa_{\bullet,0}= \fa_{c\bullet,1}$ if and only if $ c\bm{\alpha}_0 =\bm{\alpha}_1$, the result follows. 
\end{proof}

\subsection{Uniqueness of normalized volume minimizer}
Throughout this section, we assume $R$ is essentially of finite type over a field of characteristic 0.

To prove Corollary \ref{cor:uniqueness of minimizer}, we need the following proposition, which will be deduced from \cite[Theorem 3.11]{XZ20b}

\begin{prop}\label{prop:lctconvex}
Let $x\in (X,D)$ be a klt singularity. 
If $\fa_{\bullet,0}$ and $\fa_{\bullet,1}$ are  $\fm$-filtrations, then 
\[
\lct(X,D;\fa_{\bullet,t}) 
\leq 
(1-t) \cdot \lct(X,D;\fa_{\bullet,0})+ t \cdot  \lct(X,D;\fa_{\bullet,1}),
\]
for all $t\in (0,1)$.
\end{prop}
	
	\begin{proof}	
We claim $  v(\fc_\bullet)= v(\fa_{ \bullet,t} )$ for all $v\in \Val_{R,\fm}$,
where $\fc_\bullet$ is the $\fm$-filtration defined by   
$\fc_\lambda\coloneqq \sum_{i=0}^{\lceil \la \rceil} (\fa_{(\lceil \la \rceil-i)/(1-t),0 } \cap \fa_{i/t,0})$.	
Assuming the claim holds, 
\begin{multline*}
\lct(X,D;\fa_{\bullet, t} ) 
= 
\lct(X,D;\fc_\bullet)\\
\leq 
\lct(X,D;\fa_{\bullet/(1-t),0}) +  \lct(X,D;\fa_{\bullet/t,1})
= 
(1-t) \lct(X,D;\fa_{\bullet,0}) + t \lct(X,D;\fa_{\bullet,1}),
\end{multline*}
where the inequality is \cite[Theorem 3.11]{XZ20b}.	
It remains to verify the claim.
Since $\fc_{m} \subset \fa_{t,m}$ for each integer $m>0$,   $v(\fc_{\bullet}) \geq v(\fa_{t,\bullet}) $ for all $v\in \Val_{R,\fm}$.
Next fix $\mu, \nu \in \R$ such that $\mu(1-t)+\nu t = m$.
Set $r\coloneqq \mu(1-t) = m - \nu t$. For each $\ell\in \bZ_{>0}$,
 \[
 (\fa_{\mu,0} \cap \fa_{\nu,1})^\ell
 \subset 
 ( \fa_{\ell \mu,0} \cap \fa_{\ell \nu,1})
 =
(\fa_{\ell r/(1-t),1} \cap \fa_{\ell (m-r)/t,1})
\subset
 (\fa_{\lfloor \ell r \rfloor /(1-t),0} \cap \fa_{\left(\ell m -1 - \lfloor \ell r  \rfloor \right)/t,1} )
\subset 
\fc_{\ell m -1 },
\]
Using the previous inclusions, we see that 
\[
v(\fa_{\mu,0} \cap \fa_{\nu,1}) 
=
\lim_{\ell \to \infty}
\frac{v( (\fa_{\mu,0} \cap \fa_{\nu,1})^\ell)}{ \ell} 
\geq
\lim_{\ell\to \infty}
\frac{v(\fc_{\ell m-1}) }{ \ell} 
=
v(\fc_\bullet)
.\]
We can now compute 
\[
v(\fa_{\bullet,t})
=  \lim_{m\to \infty}
\frac{v(\fa_{m,t})}{m}  =  \lim_{m\to \infty} \min_{ \mu(1-t)+\nu t= m}  \Big\{   \frac{v(\fa_{\mu,0} \cap \fa_{\nu,1})}{m} \Big\} \geq v(\fc_\bullet)
.\] 
Therefore, $v(\fa_{\bullet,t})= v(\fc_\bullet)$ as desired.
\end{proof}	

\begin{proof}[Proof of Corollary \ref{cor:uniqueness of minimizer}]
Let $v_0,v_1 \in \Val_{R,\fm}$ be minimizers of $\nvol_{(X,D),x}$.
By rescaling, we may assume $A_{X,\Delta}(v_i) = 1$ and, hence, 
 $\vol(v_i)= \nvol(x,X,D)$ for $i=0,1$. We seek to show $v_0=v_1$.
 	
Consider the geodesic $(\fa_{t,\bullet})_{t\in [0,1]}$ between $\fa_{\bullet,0} \coloneqq  \fa_{\bullet}(v_0)$ and $\fa_{\bullet,1}\coloneqq  \fa_{\bullet}(v_1)$. 
For $t\in (0,1)$,
\begin{multline}\label{eq:lctleq1}
\lct(X,D; \fa_{\bullet,t})
 \leq (1-t)  \lct(X,D;\fa_{\bullet,0}) + t  \lct(X,D;\fa_{\bullet,1}) \\
 \leq (1-t)  A_{X,D}(v_0) + t  A_{X,D}(v_1)=1
,\end{multline}
where the first inequality is Proposition \ref{prop:lctconvex}
and the second uses that $v_i(\fa_{\bullet,i})=1$ by \cite[Lemma 3.5]{Blu18}. 	 
Observe that 
 \begin{multline*}
     \nvol(x,X,D)^{1/n}
\leq
\lct(\fa_{\bullet,t})  \e(\fa_{\bullet,t})^{1/n}
\leq  \e(\fa_{\bullet,t})^{1/n}  \\
\leq \big( (1-t) \e(\fa_{\bullet,0})^{-1/n}+ t\e(\fa_{\bullet,1})^{-1/n}\big)^{-1}
=\nvol(X,D,x)^{1/n}
 \end{multline*}
where there first inequality holds by \cite[Theorem 27]{Liu18}, the second by \eqref{eq:lctleq1}, and the third by Theorem \ref{thm:main theorem}(2). 
Therefore, every inequality above, in particular the third one, is an equality. 
By Theorem \ref{thm:main theorem}(3)  and Lemma \ref{lem:valuative ideals sat}, there exists $c>0$ such that $\fa_{\bullet,0}= \fa_{c\bullet,1}$,
 which translates to $\fa_{\bullet}(cv_0)= \fa_{\bullet}(v_1)$. 
 Therefore,  $cv_0=v_1$.
Since $c=A_{X,D}(cv_0)=A_{X,D}(v_1) = 1$, we see that $v_0=v_1$.
\end{proof}

	\begin{remark}
		\cite[Theorem 27]{Liu18} was originally proved for $\bQ$-coefficients, but the proof works for $\bR$-coefficients with little change. 
	\end{remark}	
	
\subsection{Minkowksi inequality}

	Lastly we prove Corollary \ref{cor:Minkowski}. In the proof, we use the following result.
		
	\begin{proposition}\cite[Theorem 10.3]{Cut21}\label{prop:Minkowski equality implies same divisorial}
Let $\fa_\bullet$ and $\fb_\bullet$ be $\fm$-filtrations with positive multiplicity. If 
$\e(\fa_\bullet\fb_\bullet)^{1/n}=
\e(\fa_\bullet)^{1/n}+\e(\fb_\bullet)^{1/n}$, then  then there exists $c>0$ such that
$v(\fa_\bullet)=c v(\fb_\bullet)$ for all $v\in \DivVal_{R,\fm}$.
\end{proposition}

\begin{proof}[Proof of Corollary \ref{cor:Minkowski}]
By \cite[Theorem 3.2]{Cut15},
\begin{equation}\label{eq:Minkowski1}
\e(\fa_\bullet\fb_\bullet)^{1/n}\leq
\e(\fa_\bullet)^{1/n}+\e(\fb_\bullet)^{1/n}
.\end{equation}
Thus, it remains to analyze when the inequality is an equality.

If \eqref{eq:Minkowski1} is an equality, then 
Propositions \ref{prop:Minkowski equality implies same divisorial}
and \ref{prop:equiv filts} imply there exists $c>0$ such that  $\widetilde{\fa}_\bullet = \widetilde{\fb}_{c \bullet}$.
Conversely, assume there exists $c>0$ such that  $\widetilde{\fa}_\bullet= \widetilde{\fb}_{c \bullet}$.
Note that
\[
v(\fa_\bullet \fb_\bullet) 
= 
v(\fa_\bullet) + v(\fb_\bullet)
= 
(c+1) v(\fb_\bullet)
\]
where the first equality uses that $\frac{ v(\fa_m \fb_m) }{m} = \frac{v(\fa_m)}{m} + \frac{v(\fb_m)}{m}$ for all $m\in \bN$ and the second is by Proposition \ref{prop:equiv filts}. 
Therefore, Proposition \ref{prop:equiv filts} implies $\widetilde{\fa_\bullet \fb_\bullet}= \widetilde{\fb}_{(c+1)\bullet}$. We  now compute
\[
\e(\fa_\bullet \fb_\bullet)^{1/n} 
=
\e(\fb_{(c+1)\bullet})^{1/n} 
=
 \e(\fb_{ c \bullet})^{1/n} + \e(\fb_\bullet)^{1/n}
=
\e(\fa_\bullet)^{1/n} + \e(\fb_\bullet)^{1/n}
,\]
where the first and third equality hold by Corollary \ref{cor:equiv equal mult}.
\end{proof}

\section{Relation to global results}\label{sec:global}

In this section, we explain the relationship between the local constructions in this paper and certain global constructions in the K-stability literature.

\medskip
Throughout, let $X$ be an $n$-dimensional normal projective variety over a field $k$ and $L$ be an ample line bundle on $X$. 
The section ring of $(X,L)$ is
\[
R(X,L):= \bigoplus_{m\in \bN} R_m := \bigoplus_{m \in \bN} H^0(X,mL)
.
\]

\subsection{Filtrations and multiplicity}
The analogue of an $\bR$-filtration of a local ring
in the global setting is the following definition  \cite[Section 1.1]{BHJ17}, which plays an important role  in the K-stability literature. 

\begin{defn}
A \emph{filtration} $\cF$ of $R(X,L)$ is the data of vector  subspaces $\cF^\la R_m \subset R_m$ for each $\la \in \bR$ and $m\in \bN$ such that
	\begin{enumerate}
		\item  $\cF^\la R_m \subset \cF^\mu R_m$ when $\la >\mu$,
		\item $\cF^\la R_m = \cF^{\la-\epsilon} R_m$ when $0 <\epsilon\ll1$,
		\item $\cF^\la R_m \cdot \cF^\mu R_n \subset \cF^{\la+\mu}$, and
		\item $\cF^{\la} R_m = R_m$  when $\la \ll 0$, and 
		\item $\cF^{\la} R_m =0 $ when $\la \gg0$.
	\end{enumerate}
A filtration $\cF$ is \emph{linearly bounded} if there exists $C>0$ such that $\cF^{Cm}R_m=0$ for all $m>0$.
\end{defn}

The data of a filtration $\cF$ can be encoded as a norm $\chi: R(X,L)\to \bR\cup \{ +\infty\}$  by  setting
\[
\chi (s) := \max \{ \la \in \bR  \, \vert\, s \in \cF^\la R_m \}
\]
when $s \in R_m$ and   $\chi(\sum_m s_m ):= \min \{ \chi(s_m ) \}$ when $\sum_m s_m \in \oplus_{m} R_m$; see \cite[Section 1.1]{BJ21} for details. Following \emph{loc. cit.}, we write $\mathcal{N}_{\bR}$ for the set of  such norms $\chi: R(X,L) \to\bR \cup \{ +\infty\}$ that arise from linearly bounded $\bR$-filtrations of $R(X,L)$.

\begin{defn}
The \emph{energy} of a linearly bounded filtration 
$\cF$ of $R(X,L)$ is
\[
E (\cF) := \lim_{m \to \infty}  \frac{ \sum_{\la \in \bR}\dim_k \gr_{\cF}^\la R_m  }{m^n / n!},
\]
where $\gr_{\cF}^\la R_m = \cF^\la R_m / \cF^{\la+\epsilon}R_m$ and $0<\epsilon\ll1$. 
The limit in the definition exists as a consequence of \cite[Theorem 5.3]{BHJ17}.
\end{defn} 

The energy measures the size of a filtration  and is an analogue of the multiplicity of an $\bR$-filtration in the global setting.
The invariant appears under different names in the K-stability literature such
as the  energy in \cite{BHJ17} and 
the $S$-invariant in \cite{Xu21}.
For a norm $\chi \in \mathcal{N}_{\bR}$ of $R(X,L)$, $\vol(\chi)$  denotes the energy of the associated filtration  in \cite{BJ21}.

\subsection{Saturation and Rees's theorem}
Rees's theorem for $\bR$-filtrations (Theorem \ref{thm:characterization for sequences with equal volume}) says that two $\bR$-filtrations $\fa_\bullet \subset \fb_\bullet$ have equal volume if and only if their saturations are equal. 
As we will explain, 
an analogue of this result was previously proven by Boucksom and Jonsson in the global  setting.

For a norm $\chi \in \mathcal{N}_{\bR}$, Boucksom and Jonsson introduce the notion of its \emph{ maximal norm} $ \chi^{\max}$ \cite[Definition 6.16]{BJ21}.
Similar to the saturation in Definition \ref{defn:saturation}, it is defined using divisorial valuations.
When ${\rm char}(k)=0$, they prove that two norms
$\chi, \chi' \in \mathcal{N}_{\bR}$ 
with $\chi \leq \chi'$ satisfy $\vol(\chi) = \vol(\chi')$
 if and only if $\chi^{\max}= \chi'^\max$ \cite[Lemma 3.11 and Theorem 6.22]{BJ21}. The proof relies on results from  non-Archimedean pluripotential   developed in \cite{BJ22}.

\subsection{Geodesics}
The notion of a geodesic between two $\bR$-filtrations (Definition \ref{d:geodesic}) is inspired by the following definition in the global setting.

\begin{defn}[\cite{BLXZ21,Reb20}]
For two linearly bounded filtrations $\cF_0$ and $\cF_1$ of $R(X,L)$ and $t\in (0,1)$, we define a filtration $\cF_t$  of $R(X,L)$ by setting 
\[
\cF_t^\la R_m= \sum_{\mu+\nu =\la} \cF_0^\mu R_m \cap \cF_1^\nu R_m
.\]
We call $(\cF_t)_{t\in [0,1]}$ the \emph{geodesic} between $\cF_0$ and $\cF_1$.
\end{defn}

This definition was introduced by the first two authors, Xu, and Zhuang in \cite[Section 3.1.2]{BLXZ21} to prove uniqueness results for certain optimal destabilizations of Fano varieties that arise from limits of K\"ahler--Ricci flow \cite[Section 3]{BLXZ21}. 
Independently, Reboulet introduced an equivalent definition  phrased in the language of  norms on the section ring, rather than filtrations, and used it to define geodesics between metrics on a line bundle in non-Archimedean pluripotential theory \cite{Reb20}.


A different and possibly more intuitive way to understand the above definition is in terms of a well-chosen basis. 
By \cite[Lemma 3.1]{AZ20}, there exists a basis $(s_1,\ldots, s_{N_m})$ of $R_m$ that is  compatible with both $\cF_0$ and $\cF_1$. 
Here, `compatible' means that $\cF_j^\la R_m$ is the span of some subset of  $(s_1,\ldots,s_{N_m})$ for each $j \in \{0,1\}$ and $\la \in \bR$.
If we set 
$
\la_{i,j} := \max \{ \la \in \bR\, \vert \, s_i \in  \cF_j^\la R_m \}$,
then a computation shows 
\[
\cF_t^\la R_m := {\rm span}( s_i \vert\, \la_{i,0}(1-t)+ \la_{i,1} t  \geq \la)
.\]
(A similar computation is made in the proof of Proposition \ref{p:measuremum}.)
Imprecisely, this  expression shows that $(\cF_{t})_{t\in [0,1]}$  interpolates between $\cF_0$ and $\cF_1$.

As shown in  \cite[Section 3.2]{BLXZ21} and \cite[Section 4.7]{Reb20b}, various functionals on the space of filtrations are convex along geodesics.
These convexity results are proven in \cite[Section 3.2]{BLXZ21} using a measure on $\bR^2$
(similarly to the proof of Theorem \ref{thm:main theorem}) 
and imply the uniqueness of minimizers of the $h$-functional on the space of valuations on a Fano variety \cite[Section 3.3]{BLXZ21}.
Thus, Theorem \ref{thm:main theorem} and Corollary \ref{cor:uniqueness of minimizer} can be viewed as local counterparts to these global results.

\newpage

\appendix
	
	\section{Okounkov bodies and geodesics}\label{sec:computation}

    
     In this appendix, we give an alternative proof of a weaker version of Theorem \ref{thm:main theorem} (see Theorem \ref{thm:appendix}) using the theory of Okounkov bodies.
     \medskip
     
    Throughout, $X$ is an $n$-dimensional normal variety over an algebraically closed field $\mathbbm{k}$ and $x\in X$ is a closed point. 
    We set $R: = \cO_{X,x}$ and write $\fm\subset R$ for the maximal ideal.
    Note that the assumptions on $R$ are more restrictive than in the earlier sections of this paper. 
    \medskip


	
		\subsection{Good valuations}
	\begin{definition}\label{defn:good valuation}\cite[Definition 8.3]{KK14}
	    We equip a total order $\preceq$ on $\bZ^n$ respecting addition.
	    Let $v: R\setminus \{0\}\to\bZ^n$ be a valuation.
		We say that $v$ is a \emph{good valuation} if it satisfies the following conditions:
		
		\begin{enumerate}
			\item $v$ has one-dimensional leaves, that is, for any $f,g\in R$ with $v(f)=v(g)$, there exists $\lambda\in \mathbbm{k}$ such that $v(f+\lambda g)>v(f)$. 
			
			\item The value semigroup $\cS=v(R\backslash\{0\})\cup \{0\}$ generates $\bZ^n$ as a group, and its associated cone $C(\cS)$, which is the closure of the convex hull of $\cS$, is a strongly convex cone. The latter is equivalent   to the existence of a linear function $\xi:\bR^n\to \bR$ such that $C(\cS)\backslash\{0\}\subset \xi_{> 0}$.
			
			\item There exist $r_0>0$ and a linear function $\xi$ as above such that 
			\[
				\ord_\fm(f)\ge r_0\xi(v(f))
			\] 
			for any $f\in R\backslash\{0\}$.
			
		\end{enumerate}
	\end{definition}

	For every quasi-monomial valuation $w\in\Val_{X,x}$, we can construct a good valuation $v$ on $X$ associated to $w$ as follows. The construction is similar to \cite[Example 8.5]{KK14}, where the characteristic of $\mathbbm{k}$ is $0$. We provide a proof for the general case here for the reader's convenience.
	
\begin{lemma}\label{lem:good valuation associated to a qm valuation}
		If $w\in \Val_{X,x}$ is a quasi-monomial valuation, 
		then there exist a vector $\bm\alpha\in \bR^n$ and a good $\bZ^n$-valued valuation $v$ on $R$ such that 
		\[
			w(f)=\langle \bm\alpha, v(f) \rangle
		\]
		for any $f\in R\backslash\{0\}$. In addition, we can take $\xi=\langle \bm\alpha,\cdot \rangle$ for condition (3) of a good valuation. 
\end{lemma}

\begin{proof}
Since $w$ is quasi-monomial, there is a birational morphism of varieties $\mu:Y\to X$, a point $\eta\in Y$, and a divisor $D$ on $Y$ that is simple normal crossing (snc) at $\eta$ such that $w\in \QM_{\eta}(Y,D)$.
We may further assume $\eta\in Y$ is a closed point and $D$ is locally defined at $\eta$ 
by a regular system of a parameters $y_1,\ldots, y_n \in \cO_{Y,\eta}$. (Indeed, since the locus where $(Y,D)$ is snc is open, we may replace  $\eta$ with closed point in $\overline{\eta}$ and possibly add components to $D$ so that the latter holds.)
By replacing $Y$ with a higher birational model as in the proof of \cite[Lemma 3.6]{JM12}, we may assume $w=v_{\bm{\alpha}}$, where ${\bm {\alpha}}:= (\alpha_1,\ldots, \alpha_r, 0, \ldots, 0)$ and $\alpha_1,\ldots, \alpha_r$ are linearly independent over $\mathbb{Q}$.



		
		Choose $\alpha'_{r+1},\ldots,\alpha'_n>0$ linearly independent over $\bQ$ and define $\bm\alpha'\coloneqq (0,\ldots,0,\alpha'_{r+1},\ldots,\alpha'_n)$. Then we can define a quasi-monomial valuation $u$ on $k[\![y_1,\ldots,y_n]\!]\cong \widehat\cO_{Y,z}$ given by the weight vector $\bm\alpha'$. Define $\bm{u}:\widehat\cO_{Y,z}\backslash\{0\}\to (\bR^2,\le_{\rm lex})$ by
		\[
			\bm{u}(f)\coloneqq (w(f),u(f)). 
		\]
		Then it is easy to check that $\bm{u}$ is an $\bR^2$-valued valuation. Moreover, it induces a well-order $\preceq$ on $\bN^n$ by 
		\[
			\bm\beta_1\preceq \bm\beta_2 \quad \text{ if and only if } \quad  (\langle \bm\alpha,\bm\beta_1 \rangle, \langle \bm\alpha',\bm\beta_1 \rangle)\le_{\rm lex} (\langle \bm\alpha,\bm\beta_2 \rangle, \langle \bm\alpha',\bm\beta_2 \rangle).
		\]
		Now define $v:\widehat\cO_{Y,z}\backslash\{0\}\to (\bN^n,\preceq)$ by $v(f)\coloneqq \min\{\bm\beta\mid c_{\bm\beta}\ne 0\}$. Then $v$ is a $\bZ^n$-valued valuation on $X$ with center $x$. 
		
		Following the calculation in \cite[Section 4]{Cut13}, we see that the value semigroup $\cS$ contains $\bm\beta'$ and $\bm\beta'+\bm{e}_i$ for some $\bm\beta'$ and $1\le i\le n$, where $\bm{e}_i\in\bZ^n$ is the $i$th standard basis vector. 
		Hence $\cS$ generates $\bZ^n$ as a group.
		
		By definition we have $w(f)=\langle \bm\alpha,v(f) \rangle$. Let $\xi\coloneqq\langle \bm\alpha,\cdot \rangle:\bR^n\to \bR$.
  Then we know that $\xi(\bm\beta)\ge \min\{\alpha_i\mid 1\le i\le r\}>0$  for any $0\ne\bm\beta=v(f)\in\cS$. So $C(\cS)\backslash\{0\}\subset \xi_{>0}$ and condition (2) is satisfied. 
		
		By Izumi's inequality Lemma \ref{lem:Izumi ineq}, there exists $r>0$ such that $w(f)\le r\cdot \ord_\fm(f)$ for any $f\in R\backslash\{0\}$. So if we set $r_0\coloneqq r^{-1}$, then 
		\[
			\ord_\fm(f)\ge r_0w(f)=r_0\xi(v(f)).
		\]
		Thus, condition (3) is satisfied. 
		Condition (1) is straightforward to verify. So we conclude $v$ is a good $\bZ^n$-valued valuation. 
	\end{proof}

	\subsection{A volume formula for convex bodies}
	
	The following lemma will be useful in proving the convexity of multiplicities. It is a slight generalization of \cite[Lemma 4.4]{Izm14} which follows from the same argument, so we omit the proof.
	
	\begin{lemma}\cite{Izm14}\label{lem:integration formula for volumes}
		Let $C\subset \bR^n$ be a strongly convex cone of dimension $n$, and $h:\Int(C)\to \bR_{\ge 0}$ a continuous function homogeneous of degree $1$. Then
		\[
		n!\vol(h_{\le 1})=\int_C e^{-h}d\mu,
		\]
		where $h_{\le 1}\coloneqq \{x\in \Int(C)\mid h(x)\le 1\}$ and $\mu$ is the Lebesgue measure on $\bR^n$.  
	\end{lemma}
	
	\subsection{The cutting function induced by a filtration}
	
	To proceed, we fix a quasi-monomial valuation $w\in\Val_{X,x}$ and a good $\bZ^n$-valued valuation $v$ associated to $w$ given by Lemma \ref{lem:good valuation associated to a qm valuation}. Let $\cS=v(R\backslash\{0\})\cup\{0\}\subset \bZ^n$ be the value semigroup of $v$. Let $\fa_{\bullet}$ be a linearly bounded $\fm$-filtration. The goal of this section is to construct a cutting function on the strongly convex cone $C(\cS)$ satisfying the conditions in Lemma \ref{lem:integration formula for volumes}. 
	
	The following lemma is an easy generalization of Izumi's inequality to filtrations. For a family version, see Lemma \ref{lem:Izumi estimate for m-function}.
	
	\begin{lemma}\label{lem:Izumi bound on value semisubgroup}
		With notation as above, for any $m\in\bR_{\ge 0}$, there exists $\bm\beta_0\in \cS$ such that if $f\in R\backslash \{0\}$ satisfies $v(f)\succeq \bm\beta_0$, then $f\in \fa_{>m}$.
	\end{lemma}

	\begin{proof}
		By Lemma \ref{lem:Izumi ineq}, there exists $r>0$ such that $\xi(v(f))=w(f)\le r\cdot \ord_\fm(f)$. By assumption, there exists $d\in\bZ_{>0}$ such that $\fm^{d}\subset \fa_1$. So if we choose $\bm\beta_0\in \cS$ such that $\xi(\bm\beta_0)\ge rd(\lceil m\rceil +1)$, then for $v(f)\succeq \bm\beta_0$, we have $\ord_\fm(f)\ge d(\lceil m\rceil +1)$.  Hence,
		\[
			f\in\fm^{d(\lceil m\rceil +1)}\subset \fa_1^{\lceil m \rceil+1}\subset \fa_{>m}. \qedhere
		\]
	\end{proof}
	
	We define a function $m_{\fa_{\bullet}}:\cS\to \bR_{\geq 0}$ associated to $\fa_{\bullet}$ by
	\begin{equation}\label{eqn:definition of m-function}
		m_{\fa_{\bullet}}(\bm\beta)\coloneqq \sup\{m\in \bR_{\geq 0}\mid v^{-1}(\bm\beta)\cap \fa_m\ne \emptyset\}
	\end{equation}
	for any $\bm\beta\in\cS$. Since $\fa_\bullet$ is left continuous, the above supremum is indeed a maximum. Now we choose an element $f_{\fa_{\bullet},\bm\beta}\in v^{-1}(\bm\beta)\cap \fa_{m_{\fa_{\bullet}}(\bm\beta)}$. We will write $m(\bm\beta)$ and $f_{\bm\beta}$ respectively, if there is no chance of ambiguity. 
	
	We first prove some properties of $m(\cdot)$ and construct a function $h$ on $\cS$ using it, which will play a key role in estimating the multiplicities. 
	
	\begin{proposition}\label{prop:volume computing function}
		Let $m=m_{\fa_\bullet}:\cS\to \bR_{\geq 0}$ be the function associated to a linearly bounded $\fm$-filtration $\fa_{\bullet}$ as above. Then the following statements hold.
		\begin{enumerate}
			\item $m$ is superadditive, that is,
			\[
				m(\bm\beta_1+\bm\beta_2)\ge m(\bm\beta_1)+m(\bm\beta_2),
			\]
			for any $\bm\beta_1,\bm\beta_2\in\cS$.
			
			\item There exists $M_2>0$ such that for any $\bm\beta\in\cS$, we have $m(\bm\beta)\le M_2\xi(\bm\beta)$.
			
			\item If $\fa_\bullet=\fa_\bullet(u)$ for some valuation $u\in\Val_{X,x}$, then we have
			\[
				u(f)\le\lim_{k\to\infty} \frac{m(kv(f))}{k},
			\]
			for all $f\in R\setminus \{0 \}$ 
			and the equality holds if $u=w$.
			
			\item There is a continuous, concave function $h_{\fa_{\bullet}}:\Int(C(\cS))\to \bR_{\ge 0}$ homogeneous of degree $1$, such that for any $\bm\beta\in\cS$,  $h_{\fa_{\bullet}}(\bm\beta)=\lim_{k\to\infty}\frac{m_{\fa_{\bullet}}(k\bm\beta)}{k}$. As before, we will drop $\fa_{\bullet}$ from the notation of $h(\bm\beta)$ if there is no ambiguity.
			
			\item For $\lambda>0$, we have
			\begin{equation}\label{eqn:rescaling of h}
				h_{\fa_{\lambda\bullet}}=h_{\fa_\bullet}/\lambda.
			\end{equation}
		\end{enumerate}
	\end{proposition}
	
	\begin{proof}
		(1) For any $\bm\beta_1,\bm\beta_2\in\cS$, we have
		\[
		f_{\bm\beta_1}f_{\bm\beta_2}\in v^{-1}(\bm\beta_1+\bm\beta_2)\cap \fa_{m(\bm\beta_1)}\fa_{m(\bm\beta_2)}\subset v^{-1}(\bm\beta_1+\bm\beta_2)\cap \fa_{m(\bm\beta_1)+m(\bm\beta_2)}.
		\]
		Hence, $m(\bm\beta_1+\bm\beta_2)\ge m(\bm\beta_1)+m(\bm\beta_2)$. 
		
		(2)	By Lemma \ref{lem:Izumi ineq} there exists $r>0$ such that 
		$\ord_\fm(f)< rw(f)=r\xi(v(f))$. Hence if we choose $f\in v^{-1}(\bm\beta)$, then we know that $\ord_\fm(f)< r\xi(\bm\beta)$, that is $f\notin \fm^{\lceil r\xi(\bm\beta)\rceil}$. Since $\fa_\bullet$ is linearly bounded by $\fm$, there exists $c>0$ such that $\fa_{ck}\subset \fm^k$ for any $k\in\bZ_{>0}$. In particular, we get $f\notin \fa_{c\lceil r\xi(\bm\beta)\rceil}$. So $m(\bm\beta)\le M_2\xi(\bm\beta)$, where $M_2\coloneqq cr$.
		
		(3) Assume that $\fa_\bullet=\fa_\bullet(u)$. Then for any $k\in\bN$, we have $f^k \in v^{-1}(kv(f))\cap \fa_{ku(f)}(u)$ which implies that 
		\[
		    m(kv(f)) \geq ku(f).
		\]
		Note that $\lim_{k\to\infty} \frac{m(kv(f))}{k}\le M_2$ exists by (1) and (2). Letting $k\to\infty$ we get the inequality. If $u=w$, then by definition we have $m(\bm\beta)=\langle \bm\alpha,\bm\beta \rangle$, which implies the equality immediately. 
		
		(4) Choose finitely generated subsemigroups $\cS_1\subset \cS_2\subset \cdots$ of $\cS$ such that $\cup_{l=1}^\infty \cS_l=\cS$ and each $\cS_l$ generates $\bZ^n$ as a group. By \cite[\S3, Proposition 3]{Kho92}, there exists $\gamma_l\in \cS_l$ such that 
		\[
		(C(\cS_l)+\gamma_l)\cap \bN^n\subset \cS_l. 
		\]
		
		For any $\bm\beta\in\Int(C(\cS_l))$, there exists $k_0$ such that for any $k\ge k_0$, we have
		\[
		\lfloor k\bm\beta \rfloor \in(C(\cS_l)+\gamma_l)\cap \bN^n\subset \cS_l.
		\]
		So we can define 
		\begin{equation}\label{eqn:definition of h}
			h(\bm\beta)\coloneqq \limsup_{k\to\infty}\frac{m(\lfloor k\bm\beta \rfloor)}{k}.
		\end{equation}
		
		We first show that this is indeed a limit. Since $\bm\beta\in\Int(C(\cS_l))$, there exists $d\in \bZ_{>0}$ such that $B(d\bm\beta,3)\subset C(\cS_l)+\gamma_l$, where we use the maximum norm on $\bR^n$. Hence for any $k_1,k_2\ge k_0$, we have
		\[
			\lfloor (k_1+k_2+d)\bm\beta \rfloor -(\lfloor k_1\bm\beta \rfloor + \lfloor k_2\bm\beta \rfloor)\in C(\cS_l)+\gamma_l\subset \cS_l. 
		\]
		By (1) we have
		\[
			m(\lfloor (k_1+k_2+d)\bm\beta \rfloor)\ge m(\lfloor k_1\bm\beta \rfloor)+m(\lfloor k_2\bm\beta \rfloor). 
		\]
		This shows that the function $k\mapsto m(\lfloor (k-d)\bm\beta \rfloor)$ is sup-additive for $k\ge k_0+d$. So by (2) and Fekete's lemma we know that there exists 
		\[
			\lim_{k\to\infty} \frac{m(\lfloor k\bm\beta \rfloor)}{k}=\lim_{k\to\infty} \frac{m(\lfloor (k-d)\bm\beta \rfloor)}{k}\le M_2.
		\] 
		Since the above definition does not depend on $l$, we have a well-defined function $h$ on $\Int(C(\cS))=\cup_{l=1}^\infty \Int(C(\cS_l))$. Clearly $h$ is homogeneous of degree $1$. 
		
		Next we prove the continuity of $h$. It suffices to prove the continuity of $h$ in $\Int(C(\cS_l))$. We first show that for any $\bm\beta,\bm\beta_1\in\Int(C(\cS_l))$, we have 
		\begin{equation}\label{eqn:h is non-decreasing}
			h(\bm\beta+\bm\beta_1)\ge h(\bm\beta).
		\end{equation}
		Indeed, there exists $k_0$ such that for any $k>k_0$, we have $B(k\bm\beta,2)\subset C(\cS_l)+\gamma_l$. Hence
		\[
			\lfloor k(\bm\beta+\bm\beta_1)\rfloor -\lfloor k\bm\beta \rfloor\in \cS_l,
		\]
		and by (1) we get $m(\lfloor (k(\bm\beta+\bm\beta_1 \rfloor))\ge m(\lfloor k\bm\beta \rfloor))$. Dividing by $k$ and letting $k\to \infty$, we get \eqref{eqn:h is non-decreasing}. Now for any $\epsilon\in(0,1)$ and $\bm\beta\in\Int(C(\cS_l))$, there exists $\rho>0$ such that $B(\epsilon\bm\beta,\rho)\subset \Int(C(\cS_l))$. Hence for any $\bm\beta'\in B(\bm\beta,\rho)$, we have $\epsilon\bm\beta\pm (\bm\beta-\bm\beta')\in \Int(C(\cS_l))$. By \eqref{eqn:h is non-decreasing} and the homogeneity of $h$, we get
		\[
			(1-\epsilon)h(\bm\beta)=h((1-\epsilon)\bm\beta)\le h(\bm\beta')\le h((1+\epsilon)\bm\beta)=(1+\epsilon)h(\bm\beta). 
		\]
		Hence, $h$ is continuous on $\Int(C(\cS_l))$. 
		
		The concavity of $h$ follows easily from its continuity, homogeneity and (1). 
		
		(5) By definition we have
		\[
			m_{\fa_\bullet}(\bm\beta)= \lambda\cdot m_{\fa_{\lambda\bullet}}(\bm\beta)
		\]
		for any $\bm\beta\in \cS$. So we have
		\[
			h_{\fa_\bullet}(\bm\beta)=\lim_{m\to\infty}\frac{m_{\fa_\bullet}(k\bm\beta)}{k}=\lambda\cdot h_{\fa_{\lambda\bullet}}(\bm\beta)
		\]
		for any $\bm\beta\in\cS$. Now \eqref{eqn:rescaling of h} follows easily from the homogeneity and continuity of $h$. 
	\end{proof}

	The following proposition asserts that $f_{\bm\beta}$ gives a basis for $\fa_m/\fa_{>m}$.

	\begin{proposition}\label{prop:compatible basis}
		For any $m\in \bR_{\geq 0}$, the set $\{[f_{\bm\beta}]\mid m(\bm\beta)=m\}$ forms a basis for the $\mathbbm{k}$-vector space $\fa_m/\fa_{>m}$. 
	\end{proposition}
	
	\begin{proof}
		We may assume that $m$ is a jumping number of $\fa_\bullet$. Let 
		\[
			\fB_m\coloneqq 
			\{\bm\beta\in\cS\mid m(\bm\beta)=m\}.
		\]
		
		We first show that the set $\{[f_{\bm\beta}]\}$ is linearly independent. Assume to the contrary that there exist $c_{\bm\beta}\in \mathbbm{k}$, not all zero, such that 
		\[
			g\coloneqq \sum_{\bm\beta\in\fB_m} c_{\bm\beta}f_{\bm\beta}\in \fa_{>m}.
		\]
		Then we have $v(g)=\bm\beta_0\coloneqq \min\{\bm\beta\in\fB_m\mid c_{\bm\beta}\ne0\}$, which implies $m(\bm\beta_0)>m$, a contradiction. 
		
		It remains to prove that for any $f\in\fa_m\backslash\fa_{>m}$, there exist $c_{\bm\beta}\in \mathbbm{k}$ for each $\bm\beta\in\fB_m$ such that 
		\begin{equation}\label{eqn:generation of f_beta}
			f-\sum_{\bm\beta\in\fB_m} c_{\bm\beta}f_{\bm\beta}\in \fa_{>m}. 
		\end{equation}
		By condition (1) of Definition \ref{defn:good valuation}, there exists a unique sequence $\{a_{\bm\beta}\}_{\bm\beta\in \cS}$ such that for any $\bm\beta'\in \cS$, we have
		\[
		w(f-\sum_{\bm\beta\preceq \bm\beta'}a_{\bm\beta}f_{\bm\beta})>\langle \bm\alpha,\bm\beta' \rangle.
		\]
		Let $\bm\beta_0$ be as defined in Lemma \ref{lem:Izumi bound on value semisubgroup}. In particular, we have $\bm\beta\prec \bm\beta_0$ for any $\bm\beta\in \fB_m$. Write
		\[
			f-\sum_{\bm\beta\in\fB_m} a_{\bm\beta}f_{\bm\beta}=(f-\sum_{\bm\beta\prec \bm\beta_0} a_{\bm\beta}f_{\bm\beta})
			+\sum_{\bm\beta\notin\fB_m,\bm\beta\prec \bm\beta_0}a_{\bm\beta}f_{\bm\beta}=:g_1+g_2
		\]
		By the choice of $a_{\bm\beta}$, we know that $v(g_1)\ge \bm\beta_0$. Hence $g_1\in \fa_{>m}$ by the choice of $\bm\beta_0$. 
		
	 \textbf{Claim:} For any $\bm\beta\preceq \bm\beta_0$ with $\bm\beta\notin\fB_m$, if $a_{\bm\beta}\ne 0$, then $m(\bm\beta)> m$. 
		
		We prove the claim by induction. Let $\bm\beta\preceq \bm\beta_0$ and assume that the claim is true for any $\bm\beta'\prec \bm\beta$. Then $f_{\bm\beta'}\in \fa_m$ for $\bm\beta'\prec \bm\beta$, and hence $f-\sum_{\bm\beta'\prec \bm\beta}a_{\bm\beta'}f_{\bm\beta'}\in\fa_m$. Now $a_{\bm\beta}\ne 0$, so we have 
		\[
			v(f-\sum_{\bm\beta'\prec \bm\beta}a_{\bm\beta'}f_{\bm\beta'})=v(a_{\bm\beta}f_{\bm\beta})=\bm\beta.
		\] 
		Hence $v^{-1}(\bm\beta)\cap\fa_m\ne \emptyset$ and  $m(\bm\beta)\ge m$. Since $m(\bm\beta)\ne m$ by assumption, we have strict inequality and the claim is proved. 
		
		By the claim we know that $g_2\in \fa_{>m}$. Thus we conclude that 
		\[
			f-\sum_{\bm\beta\in\fB_m} a_{\bm\beta}f_{\bm\beta}\in\fa_{>m},
		\]
		that is, \eqref{eqn:generation of f_beta} holds with $c_{\bm\beta}=a_{\bm\beta}$. 
	\end{proof}

	\subsection{Multiplicities and Okunkov bodies}
	
	In this section we apply the strategy as in \cite{LM09,KK14} to estimate the multiplicities.
	
	Let $\fa_{\bullet,i}$ be two linearly bounded $\fm$-filtrations for $i=0,1$. Recall that the geodesic $(\fa_{\bullet,t})_{t\in [0,1]}$ between $\fa_{\bullet,0}$ and $\fa_{\bullet,1}$ is defined as
	\[
		\fa_{\lambda,t}\coloneqq \sum_{\lambda=(1-t)\mu+ t\nu} (\fa_{\mu,0}\cap \fa_{\nu,1}).
	\]
	Then $\fa_{\bullet,t}$ is also a linearly bounded $\fm$-filtration.

	\begin{theorem}\label{thm:appendix}
	Let $w\in \Val_{X,x}$ be a quasi-monomial valuation. Denote by $\fa_{\bullet,0}:=\fa_{\bullet}(w)$.
    Let  $\fa_{\bullet,1}$ be a linearly bounded $\fm$-filtration. Let $(\fa_{\bullet,t})_{t\in[0,1]}$ be the geodesic between $\fa_{\bullet,0}$ and $\fa_{\bullet,1}$. 
    The function $E(t) \colon [0,1]\to \bR$ defined by $E(t) \coloneqq  \e(\fa_{\bullet,t})$
    satisfies the following properties:
    \begin{enumerate}
        \item $E(t)$ is smooth; 
        \item $E(t)^{-1/n}$ is concave, meaning 
        \begin{equation*}\label{eqn:app-E^1/n is concave}
            E(t)^{-1/n}\ge (1-t)E(0)^{-1/n}+tE(1)^{-1/n} \quad \text{ for all } t\in [0,1];
        \end{equation*}
        \item Suppose, in addition, $\fa_{\bullet,1} = \fa_{\bullet}(w')$ for some valuation $w'\in \Val_{X,x}$. Then $E(t)^{-1/n}$ is linear if and only if $w' = cw$ for some $c\in \bR_{>0}$.
    \end{enumerate}
    \end{theorem}

From now on, we follow the notation of Theorem \ref{thm:appendix}. Let $v$ be a good valuation associated to $w$ given by Lemma \ref{lem:good valuation associated to a qm valuation}.

	\begin{lemma}\label{lem:control for m^t by endpoints}
		Let $\fa_{\bullet,t}$ be as above for $t\in[0,1]$. Let $m^{(t)}=m_{\fa_{\bullet,t}}:\cS\to \bR_{\geq 0}$ be the function defined in \eqref{eqn:definition of m-function}. Then we have
		\[
        m^{(t)}(\bm\beta)= (1-t)m^{(0)}(\bm\beta)+tm^{(1)}(\bm\beta)
		\]
		for any $t\in [0,1]$ and any $\bm\beta\in\cS$. 
	\end{lemma}
	
	\begin{proof}
	    Fix an arbitrary element $\bm\beta\in \cS$. For simplicity, denote by $m_t:=m^{(t)}(\bm\beta)$ for $t\in [0,1]$. 
	    
	    We first show that $m_t\geq (1-t)m_0 + tm_1$.
	    By assumption, we have $w(f) = \langle\bm\alpha, v(f)\rangle$ for any $f\in R\setminus\{0\}$. Thus
	    \[
	    m_0 =m^{(0)}(\bm\beta)= \max\{w(f)\mid f\in R\setminus\{0\} \textrm{ and }v(f)=\bm\beta\} = \langle \bm\alpha, \bm\beta\rangle.
	    \]
	    Pick $g\in v^{-1}(\bm\beta)\cap \fa_{m_1,1}$. Then clearly $w(g) = \langle\bm\alpha, v(g)\rangle = \langle\bm\alpha, \bm\beta\rangle = m_0$. Thus $g\in \fa_{m_0, 0}$ which implies that $v^{-1}(\bm\beta)\cap \fa_{m_1,1} \cap \fa_{m_0,0}$ is non-empty. Since $\fa_{m_1,1} \cap \fa_{m_0,0}\subset \fa_{(1-t)m_0+tm_1,t}$, we know that $v^{-1}(\bm\beta)\cap \fa_{(1-t)m_0+tm_1,t}\neq \emptyset$ which implies $m_t\geq (1-t)m_0+tm_1$ by \eqref{eqn:definition of m-function}.
	    
	    Next, we show that $m_t\leq (1-t)m_0+tm_1$. Choose $f\in v^{-1}(\bm\beta) \cap \fa_{m_t, t}$. We may write $f=\sum_{i} f_i$ as a finite sum such that $f_i \in \fa_{\mu_i, 0}\cap \fa_{\nu_i,1}$ where $(1-t)\mu_i +t\nu_i=m_t$ for every $i$. Since $f_i\in \fa_{\mu_i,0}=\fa_{\mu_i}(w)$, we have $w(f_i)\geq \mu_i$. Hence after replacing $(\mu_i, \nu_i)$ by $(w(f_i), t^{-1}(m_t-(1-t)w(f_i))$ the assumption $f_i \in \fa_{\mu_i, 0}\cap \fa_{\nu_i,1}$ still holds. Furthermore, if we have $\mu_i=\mu_j$ for some $i\neq j$ then we may replace $(f_i, f_j)$ by $f_i+f_j$. After finitely many steps of replacements and permutations, we obtain a decomposition $f=\sum_{i=1}^{l} f_i$ such that $f_i\in \fa_{\mu_i,0}\cap \fa_{\nu_i,1}$ where $(1-t)\mu_i + t\nu_i = m_t$ and $\mu_i=w(f_i)$ for every $i$, and $\mu_1<\mu_2<\cdots<\mu_l$. Denote by $\bm\beta_i:= v(f_i)$. Since $\mu_i = \langle\bm\alpha, \bm\beta_i\rangle$, we know that $\bm\beta_1\prec \bm\beta_2\prec\cdots\prec \bm\beta_l$. As a result, we have
	    \[
	    \bm\beta = v(f) = v(\sum_{i=1}^l f_i) = v(f_1)= \bm\beta_1.
	    \]
	    Thus, $f_1\in v^{-1}(\bm\beta)\cap \fa_{\nu_1,1}$, which implies that $m_1\geq \nu_1$ by \eqref{eqn:definition of m-function}. Meanwhile, we have 
	    \[
	    \mu_1=w(f_1)=\langle\bm\alpha,\bm\beta\rangle = m_0.
	    \]Therefore,
	    \[
	     m_t= (1-t)\mu_1 + t\nu_1\leq (1-t)m_0 + tm_1.
	   , \]
	   which completes the proof.
	\end{proof}

	As an immediate corollary to Lemma \ref{lem:control for m^t by endpoints}, we get the following result.
	
	\begin{corollary}\label{cor:h_t is linear}
		Let $h_t\coloneqq h_{\fa_{\bullet,t}}:\Int(C(\cS))\to \bR_{>0}$ be defined as in \eqref{eqn:definition of h} for $t\in[0,1]$. Then we have $h_t=(1-t)h_0+th_1$. 
	\end{corollary}
	
	The following lemma can be viewed as an Izumi-type estimate in our setting of filtrations. 
	
	\begin{lemma}\label{lem:Izumi estimate for m-function}
		There exists $M>0$ depending only on $w$ and $\fa_{\bullet,1}$ such that for any $f\in R\backslash\{0\}$ with $\xi(v(f))\ge M$ and $t\in [0,1]$, we have
		\[
		f\in \fa_{M^{-1}\cdot \xi(v(f)),t}.
		\]
	\end{lemma}
	
	\begin{proof}
		Recall that we may define $\xi:\bR^n\to \bR$ in Definition \ref{defn:good valuation} by $\xi(v(f))=w(f)$. Moreover, there exists $r>0$ such that $\ord_\fm(f)\ge r\xi(v(f))$. 
		Take $d\in\bZ_{>0}$ such that $\fm^d\subset \fa_{1,i}$ for $i=0,1$. Set $M'\coloneqq 2d\cdot r^{-1}$.
		For any $f\in R$ with $\xi(v(f))\ge M'$, we have $\ord_\fm(f)\ge r\cdot\xi(v(f))\ge 2d$. Hence,
		\[
		f\in \fm^{\lceil r\xi(v(f)) \rceil}\subset \fa_{\lfloor \frac{r\xi(v(f))}{d}\rfloor,i}\subset \fa_{M'^{-1}\cdot \xi(v(f)),i},
		\]
		for $i=0,1$, where we used the inequality $\lfloor\lambda\rfloor\ge \lambda/2$ for $\lambda\ge 2$. Hence when $\lambda\coloneqq M'^{-1}\cdot \xi(v(f))\ge 2$, we have
		\[
		    f\in \fa_{\lambda,0}\cap \fa_{\lambda,1}
		    \subset \fa_{\lambda,t},
		\]
		and the lemma is proved with $M=2M'$.
		      
	\end{proof}

	In view of the above lemma, we define 
	\[
	\Gamma\coloneqq \{(\bm\beta,m)\in\cS\times \bN\mid \xi(\bm\beta)\le Mm\} \quad \text{ and }\quad \Gamma_m\coloneqq \Gamma \cap (\cS\times \{m\}) .
	\]
 For $\bm\beta \in \cS$ and  $t\in[0,1]$, recall that  $m^{(t)}(\bm\beta) := m_{\fa_{m,t}}(\bm\beta)$ and choose an element $f^{(t)}_{\bm\beta}\coloneqq f_{\fa_{m,t},\bm\beta}\in v^{-1}(\bm\beta)\cap \fa_{m^{(t)}(\bm\beta),t}$. Then we can define
	\[
	\Gamma^{(t)}\coloneqq \{(\bm\beta,m)\in\Gamma\mid f^{(t)}_{\bm\beta}\in\fa_{m,t}\} \quad\text{ and }\quad \Gamma^{(t)}_m \coloneqq 
	 \Gamma^{(t)} \cap (\cS\times \{m\}).
	\] We will use these sets to estimate the multiplicities of $\fa_{\bullet,t}$. The following property for $\Gamma^{(t)}$ is one of the main ingredients for our proof of Theorem \ref{thm:appendix}.

	\begin{proposition}\label{prop:colength by counting}
		Let $m\in \bN$. Then 
		\[
			\ell(R/\fa_{m,t})=\#(\Gamma_m\backslash\Gamma^{(t)}_m)
		\]
		for any $t\in[0,1]$. 
	\end{proposition}

	\begin{proof}
		Note that the jumping numbers of the $\fm$-filtration $\fa_{\bullet,t}$ form a discrete set as for any $\mu>\lambda>0$, $\ell(\fa_{\lambda,t}/\fa_{\mu,t})<\infty$. By Proposition \ref{prop:compatible basis}, the quotient ring $R/\fa_{m,t}$ has a basis $\{[f^{(t)}_{\bm\beta}]\mid f^{(t)}_{\bm\beta}\notin \fa_{m,t}\}$. Note that if $\xi(v(f^{(t)}_{\bm\beta}))\ge Mm$ for some $\bm\beta\in\cS$, then $f^{(t)}_{\bm\beta}\in \fa_{M^{-1}\cdot \xi(v(f^{(t)}_{\bm\beta})),t}\subset \fa_{m,t}$ by Lemma \ref{lem:Izumi estimate for m-function}. Hence for any $\bm\beta$ satisfying $f^{(t)}_{\bm\beta}\notin\fa_{m,t}$, we have $\xi(\bm\beta)=\xi(v(f^{(t)}_{\bm\beta}))\le Mm$, that is, $\bm\beta\in \Gamma_m$. Thus 
		\[
			\Gamma_m\backslash\Gamma^{(t)}_m=\{\bm\beta\mid f^{(t)}_{\bm\beta}\notin \fa_{m,t}\},
		\]
		and the proposition follows.
	\end{proof}

	We now check that the $\Gamma$ and $\Gamma^{(t)}$ are semigroups satisfying the conditions (2.3-5) of \cite{LM09}.
	
	\begin{lemma}\label{lem:LM conditions}
		Let $t\in[0,1]$. Then
		
		\begin{enumerate}
			\item $\Gamma$ and $\Gamma^{(t)}$ are semigroups,
			
			\item $\Gamma_0=\Gamma^{(t)}_0=\{0\}$,
			
			\item  there exist finitely many vectors $(v_j,1)$ spanning a semigroup $B\subset \bN^{n+1}$ such that 
			\[
			\Gamma^{(t)}\subset \Gamma\subset B,
			\]
			and
			\item if we replace $M$ by a proper multiple of it in the definition of $\Gamma$, then $\Gamma^{(t)}$ and $\Gamma$ contain a set of generators of $\bZ^{n+1}$ as a group.
		\end{enumerate}
	\end{lemma}
	
	\begin{proof}
		(1) This follows from the sup-additivity of $m(\cdot)$, Proposition \ref{prop:volume computing function}(1).
		
		(2) This is clear. 
		
		(3) Choose $b\in \bN$ such that $b\cdot\min\{\alpha_i\mid 1\le i\le r\}>M$. Then it is easy to see that $\Gamma$ is contained in the semigroup $B$ generated by $\{(\beta_1,\ldots,\beta_n,1)\mid 0\le \beta_i\le b\}$. 
		
		(4) It suffices to show that $\Gamma^{(t)}$ generates $\bZ^{n+1}$. To this end, recall that as in the proof of Lemma \ref{lem:good valuation associated to a qm valuation}, $\cS$ contains some $\bm\beta'$ and $\bm\beta'+\bm e_i$ for $1\le i\le n$. We may assume that $\xi(\bm\beta')\ge 2M$. Let $m\coloneqq \lfloor M^{-1}\cdot\xi(\bm\beta') \rfloor-1>0$. Then by Lemma \ref{lem:Izumi estimate for m-function} we have
		\[
		    f_{\bm\beta'}\in\fa_{M^{-1}\cdot\xi(\bm\beta'),t}\subset \fa_{m+1,t}\subset \fa_{m,t}
		\]
		and
		\[
		    f_{\bm\beta'+\bm e_i}\in\fa_{M^{-1}\cdot\xi(\bm\beta'+\bm e_i),t}\subset \fa_{m,t}.
		\]
		If we replace $M$ by $M'=2M$ in the definition of $\Gamma$, then
		\[
		    \xi(\bm\beta')=M'\cdot \frac{M^{-1}\cdot\xi(\bm\beta')}{2}\le M'\cdot \lfloor M^{-1}\cdot\xi(\bm\beta') \rfloor =M'm,
		\]
		where we used the inequality $\lfloor \lambda \rfloor\ge \lambda/2$ for $\lambda\ge 2$. Now by definition, $\Gamma^{(t)}$ contains the vectors $(\bm\beta',m)$, $(\bm\beta'+\bm e_i,m)$, $1\le i\le n$ and $(\bm\beta',m+1)$, which implies that $\Gamma^{(t)}$ generates $\bZ^{n+1}$ as a group.
	\end{proof}

	From now on we assume that $\Gamma$ and $\Gamma^{(t)}$ satisfy the conditions of Lemma \ref{lem:LM conditions}. As in \cite[(2.1)]{LM09}, we define $\Delta$ as the closed convex hull of $\bigcup_m \Gamma_m/m$, and define $\Delta^{(t)}$ as the closed convex hull of $\bigcup_m \Gamma_m^{(t)}/m$ for $t\in[0,1]$. Then $\Delta^{(t)}\subset \Delta$. 
	By Lemma \ref{lem:LM conditions}, Proposition \ref{prop:colength by counting} and \cite[Proposition 2.1]{LM09}, we have the following result.

	\begin{corollary}\label{cor:Okunkov body type estimate}
		With notation as above, then
		\[
				\e(\fa_{\bullet,t})=n!(\vol(\Delta)-\vol(\Delta^{(t)}))
		\]			
		for $t\in[0,1]$.
	\end{corollary}

	Next, we show that the $\Gamma^{(t)}$ can be characterized by the function $h_t$ defined as in Corollary \ref{cor:h_t is linear}.
	
	\begin{proposition}\label{prop:cut by hyperplane}
		Let $h_t\coloneqq h_{\fa_{\bullet,t}}$ for $t\in[0,1]$, where $h$ is defined by \eqref{eqn:definition of h}. Then
		\begin{enumerate}
			\item $\Int(\Delta^{(t)})=\Int(\Delta)\cap (h_t)_{>1}$, and
			
			\item $\Int(\Delta)\backslash \Delta^{(t)}=\Int(C(\cS))\cap (h_t)_{<1}$. 
		\end{enumerate}
	\end{proposition}

	\begin{proof}
		(1) First we show that $\Int(\Delta^{(t)})\subset \Int(\Delta)\cap(h_t)_{>1}$. It suffices to show that for any $\bm\beta\in\Delta^{(t)}\cap \Int(C(\cS))$ we have $h_t(\bm\beta)\ge 1$. Recall that $\Delta^{(t)}$ is the closed convex hull of $\cup_m \Gamma^{(t)}_m/m$. Hence by the concavity of $h_i$ (Proposition \ref{prop:volume computing function}(4)), it suffices to show that for any $(\bm\beta,m)\in\Gamma^{(t)}\cap \Int(C(\cS))$ we have $h_t(\bm\beta)\ge m$. But this follows immediately from \eqref{eqn:definition of h} and the definition of $\Gamma^{(t)}$.
		
		To prove the converse, assume that $\bm\beta\in \Int(\Delta)\cap (h_t)_{>1}$. Let $\epsilon\coloneqq \frac{h_t(\bm\beta)-1}{3}>0$. Applying the proof of Proposition \ref{prop:volume computing function}(4), we know that there exists $m_1$ such that for any $m\ge m_1$, there exists $\bm\beta_m\in\Gamma_m$ with $\lim_{m} \frac{\bm\beta_m}{m}=\bm\beta$. Since $h_t$ is continuous, we may choose $m_2\ge m_1$ such that for any $m\ge m_2$, we have $h_t(\frac{\bm\beta_m}{m})>1+2\epsilon$. By \eqref{eqn:definition of h}, for each $m\ge m_2$, we may choose $k_m$ such that $(1+\epsilon)m^{(i)}(k_m\bm\beta_m)\ge h_i(\bm\beta_m)$ for $i=0,1$.
    Then for $m\ge m_2$, we have
		\[
			m^{(t)}(k_m\bm\beta_m)\ge  \frac{1}{1+\epsilon}h_t(k_m\bm\beta_m)
			\ge \frac{1+2\epsilon}{1+\epsilon}k_mm> k_mm.
		\]
		This shows that $k_m\bm\beta_m\in\Gamma^{(t)}_{k_mm}$, hence
		\[
		\bm\beta=\lim_{m\to\infty} \frac{k_m\bm\beta_m}{k_mm}\in \Delta^{(t)}.
		\]

		(2) By (1) we have $\Int(\Delta)\backslash\Delta^{(t)}=\Int(\Delta)\cap (h_t)_{<1}$, so it suffices to show that for any $\bm\beta\in \Int(C(\cS))\backslash \Delta$, we have $h_t(\bm\beta)>1$. By definition we have 
		\[
			\Delta=C(\Gamma)\cap (\bR^n\times\{1\})=C(\cS)\cap \xi_{\le M}.
		\]
		Hence by Proposition \ref{prop:volume computing function} and Lemma \ref{lem:Izumi estimate for m-function}, we have
		\begin{align*}
		    h_t(\bm\beta)=&\lim_k \frac{m^{(t)}(\lfloor k\bm\beta \rfloor)}{k}\ge \lim_k\frac{M^{-1}\cdot \xi(\lfloor k\bm\beta \rfloor)}{k}\\
		        =&M^{-1}\cdot\xi(\bm\beta)>1,
		\end{align*}
		where in the first inequality and the second equality, we used the fact that for any $k\in\bZ_{>0}$,
		\[
		    k\xi(\bm\beta)\ge \xi(\lfloor k\bm\beta \rfloor)=\langle \bm\alpha,\lfloor k\bm\beta \rfloor \rangle\ge k\xi(\bm\beta)-\|\bm\alpha\|_{1}.
	    \] 
	    The proof is finished.  
	\end{proof}

	Next we prove Theorem \ref{thm:appendix}.
	
	\begin{proof}[Proof of Theorem \ref{thm:appendix}]
		By Lemma \ref{lem:integration formula for volumes}, Corollary \ref{cor:Okunkov body type estimate} and Proposition \ref{prop:cut by hyperplane}, we have
		\begin{equation}\label{eqn:volume formula}
			E(t)=n!\vol(\Delta\backslash\Delta^{(t)})=n!\vol(\Int(C(\cS))\cap (h_t)_{<1})=\int_{C(\cS)} e^{-h_t}d\mu.
		\end{equation}
		By Corollary \ref{cor:h_t is linear} we have $h_t=(1-t)h_0+th_1$, hence $\frac{d}{dt}h_t=h_1-h_0$. By Proposition \ref{prop:volume computing function}(4), we can differentiate under the integral sign, and hence $E(t)$ is a smooth function. Thus (1) is proved.

		To prove \eqref{eqn:app-E^1/n is concave}, we use a slight variant of the formula \eqref{eqn:volume formula}. Let $h\coloneqq h_1-h_0$. Then by homogeneity we have
		\begin{align*}
			E''(t)=&\int_{C(\cS)} h^2e^{-h_t}d\mu\\
				=&\int_0^\infty d\lambda\int_{C(\cS)\cap (h_t)_{=\lambda}}h^2e^{-h_t}d\mu_{n-1}\\
				=&\int_0^\infty \lambda^{n+1}e^{-\lambda}d\lambda \int_{C(\cS)\cap (h_t)_{=1}}h^2d\mu_{n-1}\\
				=&(n+1)!\int_{C(\cS)\cap (h_t)_{=1}}h^2d\mu_{n-1},
		\end{align*}
		where $\mu_{n-1}$ is the Lebesgue measure on $\bR^{n-1}$. Similarly we have
		\[
			E'(t)=n!\int_{C(\cS)\cap (h_t)_{=1}}hd\mu_{n-1}
		\]
		and
		\[
			E(t)=(n-1)!\int_{C(\cS)\cap (h_t)_{=1}}d\mu_{n-1}.
		\]
		Hence by Cauchy-Schwarz we have
		\begin{equation}\label{eqn:C-S}
			E''(t)E(t)\ge \frac{n+1}{n}E'(t)^2\ge E'(t)^2,
		\end{equation}
		which implies \eqref{eqn:app-E^1/n is concave}. 
		
		To prove (3), first assume that the equality in \eqref{eqn:app-E^1/n is concave} holds. By\eqref{eqn:C-S} and Cauchy-Schwarz we know that $h_1-h_0=C_t$ is a constant on $(h_t)_{=1}$. In particular, $h_1-h_0=C_0$ on $(h_0)_{=1}$. Since $h_1$ and $h_0$ are $1$-homogeneous and positive on $C(\cS)$, we know that $h_0=ch_1$, where $c:=1/(C_0+1)>0$. By Proposition \ref{prop:volume computing function}(5), we know that
		\begin{equation}\label{eqn:h is proportional}
			h_0=h_{\fa_{\bullet,0}}=h_{\fa_{c^{-1}\bullet,1}}.
		\end{equation}
		By Lemma \ref{lem:integration formula for volumes}, Corollary \ref{cor:Okunkov body type estimate} and Proposition \ref{prop:cut by hyperplane} we have 
		\[
			\vol(w)=\e(\fa_{\bullet,0})=\e(\fa_{c^{-1}\bullet,1})=\vol(cw').
		\]
		Moreover, for any $f\in R\backslash\{0\}$, let $\bm\beta:=v(f)$. Then by Proposition \ref{prop:volume computing function}(3) we have $h_0(\bm\beta)=w(f)$ and $h_{\fa_{c^{-1}\bullet,1}}(\bm\beta)\ge cw'(f)$. Combined with \eqref{eqn:h is proportional} we get $cw'(f)\le w(f)$. So by \cite[Proposition 2.7]{LX16}, we conclude that $w=c w'$.
		
		Conversely, if $w=cw'$, then for any $\lambda\in\bR_{>0}$, by a direct calculation we have
		\[
		    \fa_{\lambda,t}=\sum_{(1-t)\mu+t\nu=\lambda}\fa_{\mu}(cw')\cap\fa_{\nu}(w')=\fa_{f(t,\lambda)}(w'),
		\]
		where $f(t)=\frac{\lambda}{t+c(1-t)}$. So 
		$E(t)=(\frac{1}{t+c(1-t)})^n\vol(w')$, and the proof is finished. 
	\end{proof}

%

\end{document}